\documentclass[12pt]{amsart}
\usepackage{amsmath,amstext, amsfonts,amssymb}
\usepackage[T2A]{fontenc}
\usepackage[utf8]{inputenc}
\usepackage[russian, english]{babel}
\usepackage{cite, color}
\usepackage{appendix}
\usepackage{ulem}
\newcommand{\RR}{\mathbb R}
\newcommand{\NN}{\mathbb N}

\newcommand{\DerH}{\mathcal{D}_H}
\newcommand{\DiffH}{-_H}

\usepackage{mathrsfs}
\newtheorem{theorem}{Theorem}
\newtheorem{lemma}{Lemma}
\newtheorem{remark}{Remark}
\newtheorem{definition}{Definition}
\newtheorem{corollary}{Corollary}
\newtheorem{assumption}{Assumption}

\allowdisplaybreaks

\begin{document}

\title[Korneichuk--Stechkin Lemma, inequalities, and optimal recovery problems]{Korneichuk-Stechkin Lemma, Ostrowski and Landau inequalities, and optimal recovery problems  for  $L$-space Valued Functions}


\author{Vladyslav Babenko, Vira Babenko and Oleg Kovalenko}

\address{Vladyslav Babenko, Department of Mathematical Analysis and Theory of Functions, Oles Honchar Dnipro National University, Dnipro, Ukraine}
\email{babenko.vladislav@gmail.com}

\address{Vira Babenko, Department of Mathematics and Computer Science, Drake University, Des Moines, USA}
\email{vira.babenko@drake.edu}

\address{Oleg Kovalenko, Department of Mathematical Analysis and Theory of Functions, Oles Honchar Dnipro National University, Dnipro, Ukraine}
\email{olegkovalenko90@gmail.com}


\dedicatory{Dedicated to the memory of N.~P.~Korneichuk and S. B. Stechkin\\ on the occasion of their centenary}                    



\begin{abstract}
        We prove an analogue of the Korneichuk--Stechkin lemma for functions with values in $L$-spaces. As applications, we obtain sharp Ostrowski type inequalities and solve problems of optimal recovery of identity and   convexifying operators, as well as the problem of integral recovery,  on the classes of $L$-space valued functions with given majorant of modulus of continuity.
  The recovery is done based on  $n$ mean values of the functions  over some intervals. Moreover, on the classes of functions with
  given majorant of modulus of continuity of their Hukuhara type derivative, we solve the problem of optimal recovery of the function and the Hukuhara type derivative.
     The recovery is done based on $n$ values of the function. We also obtain some sharp Landau type inequalities and solve an analogue of the Stechkin problem about approximation of unbounded operators by bounded ones  and the problem of optimal recovery of an unbounded operator on a class of elements, known with error.
Consideration of $L$-space valued functions gives a unified approach to solution of the mentioned above extremal problems for the classes of multi- and fuzzy-valued functions as well as for the classes of functions with values in Banach spaces, in particular random processes, and many other classes of functions.

\end{abstract}
\maketitle


\section{Introduction}
Let $\omega$ be a modulus of continuity, i.~e. a non-decreasing continuous semi-additive function such that $\omega(0) = 0$. For a segment $[a,b]\subset \RR$ denote by $H^\omega([a,b],\mathbb{R})$ the class of functions $f\colon [a,b]\to \RR$ such that $|f(t)-f(s) |\leq \omega(|t-s|)$ for all $t,s\in [a,b]$. 
The moduli of continuity $\omega(\cdot)$, as independent functions with mentioned above properties, the classes $H^\omega([a,b],\mathbb{R})$, as well as the classes $W^rH^\omega([a,b],\mathbb{R})$, were introduced by Nikol'skii in~\cite{Nikolsky46}. 
For two {positive almost everywhere} and integrable functions $\psi_1\colon [a,a']\to\RR_+$ and $\psi_2\colon  [b',b]\to\RR_+, \;(a<a'\le b'<b)$, the Korneichuk--Stechkin lemma, see~\cite[\S~7.1]{ExactConstants}, gives an estimate for the functional
\begin{equation}\label{numericFunct}
(\psi_1,\psi_2)\to\sup\limits_{f\in H^\omega ([a,b],\RR)}\left|\int_a^{a'}f(t)\psi_1(t)dt-\int_{b'}^{b}f(t)\psi_2(t)dt\right|,
\end{equation}
which is sharp in the case of concave modulus of continuity $\omega$. 
This lemma was published in~\cite{Korneichuk59, Korneichuk61}, (see also a remark in~\cite{Korneichuk59}) for the classes 
$H^\omega([a,b],\RR)$ with $\omega(t)=t^\alpha$, $0<\alpha\le 1$, and was generalized to the case of arbitrary modulus of continuity in~\cite{Korneichuk62}. The Korneichuk--Stechkin lemma played an important role in the solution of many extremal problems of approximation theory, see~\cite[Chapter~7]{ExactConstants}
and references therein. Some of its generalizations and more applications can be found in~\cite{bagdasarov2012,stepanets2018}.

The theory of Banach space valued, multi-valued and fuzzy-valued functions was actively developed over the last several decades (see~\cite{Aubin,Borisovich,Diamond2}), in particular, due to its applications in optimization theory, approximation theory, mathematical economics, numerical analysis and other branches of applied mathematics. Some results on approximation of multi- and fuzzy- valued functions can be found in~\cite{nira2014,anastassiou2010}.

 Banach spaces, spaces of sets and spaces of fuzzy sets belong to the class of so-called $L$ spaces (i.~e. semi-linear metric spaces with two additional axioms, which connect the metric with the algebraic operations). The notion of an $L$-space was introduced in~\cite{Vahrameev}, see also~\cite{Aseev}. In Section~\ref{s::LSpace} we present necessary definitions and facts related to $L$-spaces. In particular, for the sake of completeness, we present the definition and some properties of the Lebesgue integral for bounded $L$-space valued functions.

In  Section~\ref{s::KSLemma} we generalize the Korneichuk--Stechkin lemma to the case of $L$-space valued functions. 
Let $(X,h_X)$ be an $L$-space and $H^\omega([a,b],X)$ be the class of functions $f\colon [a,b]\to X$ 
such that $h_X(f(t'),f(t''))\le \omega(|t'-t''|)$ for all  $t',t''\in [a,b]$. Let also $\psi_1\colon[a,a']\to \RR$ and $\psi_2\colon [b',b]\to \RR$, $a<a'\le b'<b,$ be positive almost everywhere, measurable, and bounded functions. We obtain, see Lemma~\ref{l::KSLemma}, an estimate for the functional
\begin{equation}\label{LspacevaluedFunct}
S(\psi_1,\psi_2):=\sup\limits_{f\in H^\omega ([a,b],X)}h_X\left(\int_a^{a'}f(t)\psi_1(t)dt,\int_{b'}^{b}f(t)\psi_2(t)dt\right),
\end{equation}
which is sharp in the case of concave modulus of continuity $\omega$.
 In a series of applications that we consider in this article, we show that our generalization may be an important tool for solution of extremal problems involving $L$-space valued functions.

In Section 4 we obtain a general estimate of the functional~\eqref{LspacevaluedFunct} for rather arbitrary functions $\psi_1$ and $\psi_2$ in terms of the Korneichuk $\Sigma$-rearrangement of the function $\Psi(t)=\int_a^t(\psi_1(u)-\psi_2(u))du$. This estimate generalizes the estimate for the functional~\eqref{numericFunct}, obtained by Korneichuk in~\cite{Korneichuk71}, see also~\cite[Theorem~7.1.9]{ExactConstants}. 

In 1938 Ostrowski~\cite{Ostrowski38} proved  a sharp inequality that estimates the deviation of a value of a function  from its mean value using the uniform norm of the function's derivative.
Such inequalities 
have been intensively studied, see~\cite{Dragomir17} for a survey of the obtained results. 
It is worth noting, that the general estimate for the functional~\eqref{numericFunct}, which was obtained by Korneichuk some~50 years ago, essentially contains a series of results on the Ostrowski type inequalities, which were obtained much later. In particular, from this estimate, one can easily obtain the main result from~\cite{barnet} and one of the results in~\cite{Guessab02}. In~\cite{Barnett02,Chalco12, Chalco15, Anastassiou03,Roman18,Zhao19,Barnett01,Dragomir03,Anastassiou12} such type of inequalities 
are investigated for non-real-valued functions. The obtained in this article estimate for functional~\eqref{LspacevaluedFunct} implies some of the main results in  papers~\cite{Barnett02,Anastassiou03, Anastassiou12, Chalco12}.

An important part of approximation theory and optimal algorithms theory is theory of optimal recovery of operators. 
Statements of the problems of this theory, many results and further references can be found in monographs~\cite{TraubWozhniakowski, Zhensikbaev03}.
We consider the optimal recovery problem in the following statement.
 
Let a metric space $(X, h_X)$, sets  $Z$, $Y$, a class of elements $W\subset Z$, as well as mappings $\Lambda \colon Z\to X$ and $I\colon W\to Y$  be given. We call an arbitrary mapping $\Phi\colon Y\to X$ a method of recovery of the mapping $\Lambda$ on the class $W$ based on the information given by the mapping $I$. The error of recovery of the mapping $\Lambda$ on the class $W$ by the method $\Phi$ based on the information given by the mapping $I$ is given by the formula
$$
{\mathcal E}(\Lambda,W,I,\Phi, X)={\sup\nolimits_{z\in W}h_X(\Lambda(z), \Phi(I(z)))}.
$$
The quantity 
\begin{equation}\label{errorOfRecovery}
    {\mathcal E}(\Lambda,W, I, X)=\inf\nolimits_{\Phi} {\mathcal E}(\Lambda,W,I,\Phi, X)
\end{equation}
is called the optimal error of recovery of the mapping $\Lambda$ on the class $W$ based on the information given by the mapping $I$.
The problem of optimal recovery  of the mapping $\Lambda$ on the class  $W$ with the information given by $I$ in the metric of the space $X$ is to find quantity \eqref{errorOfRecovery}
and a method $\Phi^*$ (if such a method exists) on which the infimum in the right-hand side of \eqref{errorOfRecovery} is attained.
If $\mathcal{I}$ is some class of information operators, then it is also of interest to find the quantity
$$
    {\mathcal E}(\Lambda,W, {\mathcal{I}}, X)=\inf\nolimits_{I\in\mathcal{I}} {\mathcal E}(\Lambda,W,I, X)
$$
and the best information operator.

In Section~\ref{s::HOmegaRecovery} we consider the problem of optimal recovery of the convexifying operator (see Section~\ref{s::LSpace} for definitions)
and of the integral on the class $H^\omega([a,b],X)$. Under some additional assumptions (in particular in the case of Banach space valued functions), the convexifying operator turns into the identity operator.
For real-valued functions, these problems are well studied when the informational operator maps a function from the class to its values at $n$ points of the segment $[a,b]$. Regarding recovery of a function we refer to~\cite[Chapter~5,6]{SplinesInApproxTh}; regarding the recovery of the integral we refer to~\cite{Korneichuk68}. In~\cite{Babenko15} the problem of optimization of  approximate integration was solved for the class of multi-valued functions.
Here as informational operators, we use the ones that map functions from the class to their mean values on $n\in\NN$ intervals belonging to $[a,b]$. This kind of information operators is of interest, since the analog measuring devices give such mean values of the measured functions. Moreover, the results on optimal recovery given such type of information, easily imply corresponding results on optimal recovery for the case, when the information operators map functions to their values at $n$ points of the interval $[a,b]$. 
The problem of optimization of approximate integration given the ''interval'' information for the functions from the class $H^\omega([a,b],\RR)$ was solved in~\cite{Borodachov98}.
Since  a random process can be viewed as a function into a Banach space of random variables, our results can be applied to recovery problems for random processes. Some results in this direction can be found in~\cite{Drozhina, Kumar, Kovalenko20}.

In Section~\ref{s::PolylineApproximation} we consider the problem of optimal recovery of the identity operator and the operator $\DerH$ of Hukuhara type derivative on the class $W^1H^\omega([a,b],X)$ (see Section~\ref{s::PolylineApproximation} for the definitions).
In these problems the recovery is done based on the  information operator that maps a function to its values at $n$ points of the interval $[a,b]$.
We again refer to~\cite[Chapter~5,6]{SplinesInApproxTh} for the results on optimal recovery of functions on the class $W^1H^\omega([a,b],\RR)$ and its periodic analog, as well as on optimal recovery of the derivative of the functions on these classes. 

In Section~\ref{s::StechkinPr}, we obtain sharp inequalities of Landau type for divided differences of Hukuhara type as well as for derivatives of Hukuhara type of the functions from the classes $\overline{W}^1H^\omega([a,b],X)=\bigcup_{k>0}k\cdot W^1H^\omega([a,b],X)$. 
Many known results on Landau and Landau--Kolmogorov type inequalities can be found in ~\cite{BKKP, Mitrinovich, Arestov}. For the functions with values in Banach spaces, some inequalities were obtained in~\cite{ANASTASSIOU2012312,XIAO}; 
for the $L$-spaces valued functions defined on $\RR$ or $\RR_+$ --- in~\cite{VeraBabenko_JANO}.

Inequalities of such type are intimely connected to the Stechkin problem about approximation of an operator by the ones with smaller norm, in particular approximation of unbounded operators by bounded ones. The problem was first stated in~\cite{Stechkin}, where the first results on the solution of the problem were obtained. Information on further results can be found in~\cite{Arestov,BKKP}. In~\cite{VeraBabenko_JANO}, a generalization of the Stechkin problem for the case of unbounded operators acting in $L$-spaces was proposed; some results about approximation of Hukuhara type derivatives by Lipschitz operators on the classes $W^1H^\omega(J,X)$, where $J=\RR$ or $J=\RR_+$ were obtained. Here we consider this problem for the operator of Hukuhara type divided difference and the Hukuhara type derivative. We also consider the problem of optimal recovery of the operator $\DerH$ on the class $W^1H^\omega([a,b],X)$ in the case, when the elements of the class are known with error. Known results and further references can be found in \cite{Arestov,BKKP,VeraBabenko_JANO}.

\section{$L$-spaces}\label{s::LSpace}
\subsection{Definitions}

\begin{definition}
A set $X$ is called a semilinear space, if  operations of addition of elements and their multiplication on real numbers are defined in $X$, and the following conditions are satisfied for all $x,y,z\in X$ and $\alpha,\beta\in\mathbb{R}$:
\begin{gather*}
x+y=y+x;
\\ x+(y+z)=(x+y)+z;
\\ \exists\; \theta\in X\colon x+\theta=x;
\\ \alpha(x+y)=\alpha x +\alpha y;
\\ \alpha(\beta x)=\left(\alpha\beta\right)x;
\\ 1\cdot x=x,\; 0\cdot x=\theta.
\end{gather*}
\end{definition}

\begin{definition}
We call an element $x\in X$ convex, if for all $\alpha,\beta\geq 0$, $(\alpha + \beta)x=\alpha x+ \beta x.$
Denote by $ X^{\rm c}$ the subspace of all convex elements of the space $X$.
\end{definition}
\begin{remark}
Some authors (see e.~g.~\cite{Borisovich}) include into the axioms of a semi-linear space the requirement $X=X^{\rm c}$.
\end{remark}

\begin{definition}
A  semilinear space $X$, endowed with a metric $h_X$, is called an $L$-space, if it is complete and separable and for all $x,y,z\in X$, and $\alpha\in\mathbb{R}$
$$
h_X(\alpha x,\alpha y)=|\alpha|h_X(x,y);
$$
\begin{equation}\label{ax::LSemiIsotropic}
h_X(x+z,y+z)\leq h_X(x,y).
\end{equation}
\end{definition}

\begin{remark}
It follows from the triangle inequality and~\eqref{ax::LSemiIsotropic}, that 
$$
 \forall x,y,z,w\in X\;\;\; h_X(x+z,y+w)\leq h_X(x,y)+h_X(z,w). 
$$

\end{remark}

\begin{definition}
An $L$-space $X$ is called isotropic, if inequality~\eqref{ax::LSemiIsotropic} turns into equality for all $x,y,z\in X$.
\end{definition}

Next we list some of the examples of $L$-spaces. More details can be found in~\cite{babenko19}. 
Arbitrary separable Banach space and arbitrary complete and separable quasilinear normed space (see~\cite{Aseev}) are $L$-spaces. The space  $\Omega(X)$ of non-empty compact subsets of a separable Banach space $X$ endowed with usual Hausdorff metric, the space $\Omega_{\rm conv}(X)$ of convex elements from $\Omega(X)$, and spaces of fuzzy sets (see e.~g.~\cite{Diamond2}) are also examples of $L$-spaces. All $L$-spaces mentioned above are isotropic.

An example of a non-isotropic $L$ space can be built as follows. Let $X = [0,\infty)$, for $\lambda\in\RR$, $x,y\in X$ set $x\bigoplus y = \max\{x,y\}$, $\lambda\bigodot x = |\lambda|x$. Then $(X,\bigoplus,\bigodot)$ with the metric $h_X(x,y) = |x-y|$, $x,y\in X$, is a non-isotropic $L$-space.

A function $f\colon [a,b]\to X$ is said to be measurable, if 
for any element $x\in X$ the real-valued function $h_X(f(t),x)$ is measurable.
For $[a,b]\subset \RR$ and an $L$-space $(X,h_X)$, denote by $C([a,b],X)$ and  $B([a,b],X)$ the spaces of continuous (resp. bounded and measurable) functions $f\colon [a,b]\to X$ with the metrices 
$$h_{C([a,b],X)}(f,g):= \max\limits_{t\in [a,b]} h_X(f(t),g(t)) \;\;\text{and}\;\;
h_{B([a,b],X)}(f,g):= \sup\limits_{t\in [a,b]} h_X(f(t),g(t)).$$

\subsection{Hukuhara type derivative}
The notion of the Hukuhara difference of two sets was introduced in~\cite{Hukuhara}. 
\begin{definition}
Let $X$ be an $L$-space. We say that $z\in X$ is the Hukuhara type difference of $x,y\in X$, if $x=y+z$.
We denote this difference by $z=x\DiffH y.$
\end{definition}
Note, that in an isotropic $L$-space the Hukuharu difference $x\DiffH y$ is unique, provided it exists. On the other hand, in a non-isotropic $L$-space, uniqueness is not guaranteed. For example, in the space $(X,\bigoplus,\bigodot)$, for arbitrary $x\in X$ the difference $x\DiffH x$ exists and is not unique for each $x\neq 0$. 
Everywhere below, when we consider Hukuharu differences, we assume that the $L$-space is isotropic.
\begin{definition} 
If $t\in(a,b)$, and for all small enough $\gamma>0$ there exist differences $f(t+\gamma)\DiffH f(t)$ and $f(t)\DiffH f(t-\gamma)$, and both limits $\lim\limits_{\gamma\to +0} \gamma^{-1}(f(t+\gamma)\DiffH f(t))$ and $\lim\limits_{\gamma\to +0} \gamma^{-1}(f(t)\DiffH f(t-\gamma))$ exist and are equal to each other, then the function $f$ has a Hukuhara type derivative $\DerH f(t)$ at the point $t$ (if $t=a$ or $t=b$ then there exists only one limit) and
$$
\DerH f(t):=\lim\limits_{\gamma\to +0} \gamma^{-1}(f(t+\gamma)\DiffH  f(t)).
$$
\end{definition}
One can find properties of Hukuhara type differences and elements of calculus based on Hukuhara type difference and derivative in $L$-spaces in~\cite{babenko19}.
\subsection{Integration in $L$-spaces}
For completeness we present the definition and some of the properties of the Lebesgue integral for the functions $f\in B([a,b],X)$, where $X$ is an $L$-space (see~\cite{Vahrameev} and~\cite[\S 5]{Aseev}). 
First, we recall the definition of a convexifying operator.
\begin{definition}\label{def::convexifyingOperator}
A surjective operator $P\colon X\to X^{\rm c}$ is called convexifying, if 
\begin{gather*}
h_X(P(x),P(y))\le h_X (x,y) \text{ for all } x,y\in X; \\
P\circ P = P\notag; \\
P(\alpha x + \beta y) = \alpha P(x)+ \beta P(y)\text{ for all }x,y\in X \text{ and }\alpha,\beta\in\RR.
\end{gather*}
\end{definition}

The operator  ${\rm conv}\colon\Omega(\mathbb{R}^m)\to \Omega(\mathbb{R}^m)$ that maps each $x\in \Omega(\mathbb{R}^m)$ to its convex hull ${\rm conv}\,x$ is an example of a  convexifying operator.

A mapping $f$ is called simple, if it has a finite number of values $\left\{f_k\right\}_{k=1}^n$ on pairwise disjoint measurable sets $\left\{T_k\right\}_{k=1}^n$, $n\in\NN$. The Lebesgue integral of a simple mapping $f$ is by definition
$$
\int_a^b f(s)ds :=\sum\nolimits_{i=1}^n P(f_i)\mu(T_i),
$$
where $\mu$ is the Lebesgue measure. 
The following properties hold for simple  $f,g$.
\begin{enumerate}
\item For all $\alpha,\beta\in \RR$
\begin{equation}\nonumber
\int_a^b \left(\alpha f(t)+\beta g(t)\right) dt=\alpha \int_a^b f(t) dt+\beta \int_a^b g(t) dt.
\end{equation}

\item The function $t\to h_X \left(f(t), g(t)\right)$ is integrable and
    \begin{equation}\nonumber
    h_X\left(\int_a^b f(t) dt, \int_a^b g(t) dt\right)\leq \int_a^b h_X \left(f(t),g(t)\right)dt.
    \end{equation}

\item The function $P(f(\cdot))$ is integrable and
\begin{equation}\nonumber
\int_a^b f(t) dt = P\left( \int_a^b f(t) dt\right) = \int_a^b P(f(t)) dt.
\end{equation}
\item For disjoint measurable sets $T_1$ and $T_2$ such that $[a,b]=T_1\cup T_2$
$$
\int_a^b f(t) dt=\int_{T_1} f(t) dt+\int_{T_2} f(t) dt.
$$
\end{enumerate}

Any function $f\in B([a,b],X)$ is a uniform limit of a sequence $\{f^k\}$ of simple functions.  Using standard arguments, one can prove that the sequence $\left\{\int_a^bf^k(t)dt\right\}_{k\in\NN}$ is fundamental. By definition, the integral $\int_a^bf(t)dt$ of the function $f$ is set to be the limit of this sequence.

 It is clear that properties 1
 -- 4 of the Lebesgue integral for simple functions hold for arbitrary functions from $B([a,b],X)$.
 Moreover, if $\rho$ is an absolutely continuous strictly monotone function from $[c,d]\subset \RR$ onto $[a,b]\subset\RR$, then for all $f\in B([a,b], X)$,
 $$\int_{a}^bf(t)dt = \int_{c}^df(\rho(s))\rho'(s)ds.
 $$
 Indeed, from Properties~1--4 and the possibility to change variables in the integral for real-valued functions, it follows, that this property holds for the case, when $f$ is simple. The general case can be obtained using the limiting procedure.

 Note, that in the case of $X$ being a Banach space, the integral becomes the Bochner integral, see~ \cite[Sections~3.7-3.8]{hille}; in the case $X=\Omega( {\RR^m})$, the integral coincides with the Aumann integral, see~\cite[Theorem~12]{Aseev}.

\subsection{Some Properties of L-spaces}
\begin{definition}
We say that an element $x\in X$ is invertible, if there exists an element $x'\in X$ such that $x+x'=\theta$. In this case the element $x'$ is called the inverse to $x$. Denote by $X^{\rm inv}$ the set of all invertible elements of the space $X$. 
\end{definition}

\begin{assumption}
In what follows we assume that $X^{\rm inv}\cap X^{\rm c}\neq\{\theta\}$.
\end{assumption}
In the space $\Omega(X)$ any element of the form $\{x\},$ $x\in X$, is convex and invertible. 

We need the following lemmas, see~\cite{VeraBabenko_JANO}.

\begin{lemma}
If $x\in X^{\rm inv}$, then its inverse element $x'$ is unique.
\end{lemma}


\begin{lemma}
If $x\in X^{\rm inv}\cap X^{\rm c}$, then $x'\in X^{\rm c}$. 
\end{lemma}

\begin{lemma}\label{l::distBetweenConvexElemMultipliers}
For all $x\in X^c$ and  $\alpha, \beta\in \mathbb{R}$, $
h_X(\alpha x, \beta x)\leq |\alpha -\beta |\cdot h_X (x,\theta).
$
If $X$ is isotropic and $\alpha\cdot\beta\geq 0$, then the inequality becomes equality.
\end{lemma}

In addition, we also need the following lemmas.  We omit their elementary proofs.
\begin{lemma}\label{l::distBetweenInverseElems} 
Let $X$ be an isotropic $L$-space. Then for any $ x \in X^c\cap X^{\rm inv}$, 
$$
h_X(x,x')=h_X(x+x,\theta)=2h_X(x,\theta).
$$
\end{lemma}
 \begin{lemma}\label{l::inversElementNorm} 
For any $x \in X^{\rm inv}\cap X^{\rm c}$, $h_X(x',\theta) = h_X(x,\theta)$.
\end{lemma}
\subsection{Auxiliarly results}
\begin{definition}
For $f\colon [a,b]\to\RR$ and $x\in X^{\rm inv}$ define the function $f_x\colon  [a,b]\to X$,
\begin{equation}\label{simpleFunc}
f_x(t) = f_+(t)\cdot x + f_-(t)\cdot x',
\end{equation}
where for real $\xi$, $\xi_\pm :=\max\{\pm \xi ,0\}$.
\end{definition}

\begin{lemma}\label{l::LspaceValuedFunc}
Let $X$ be an isotropic $L$-space and $f\in H^\omega ([a,b],\RR)$. If $x\in X^{\rm inv}\cap X^{\rm c}$ is such that $h_X(x,\theta) = 1$, then   $f_x\in H^\omega([a,b],X)$ and 
\begin{equation}\label{integralOfSimpleFunc}
\int_a^b f_x(t)dt  = \int_a^b f_+(t)dt\cdot x + \int_a^b f_-(t)dt\cdot x'= \left(\int_a^b f(t)dt\right)_+\cdot x + \left(\int_a^b f(t)dt\right)_-\cdot x'.
\end{equation}
\end{lemma}
\begin{proof}
If $s,t\in [a,b]$ are such that $f(s),f(t)\geq 0$, then due to Lemma~\ref{l::distBetweenConvexElemMultipliers},
$$
h_X(f_x(s), f_x(t))
=
h_X(f(s)\cdot x, f(t)\cdot x)
\leq \omega(|s-t|).
$$
Analogously, due to Lemma~\ref{l::inversElementNorm}, in the case, when  $f(s),f(t)\leq 0$. When $f(s)\geq 0\geq f(t)$, 
\begin{multline*}
h_X(f_x(s), f_x(t))
=
h_X(f_+(s)\cdot x, f_-(t)\cdot x')
= 
h_X(f_+(s)\cdot x +f_-(t)\cdot x, \theta) 
\\=
|f(s)-f(t)|h_X(x,\theta)
\leq \omega(|t-s|).
\end{multline*}
Hence $f_x\in H^\omega([a,b],X)$. Equality~\eqref{integralOfSimpleFunc} follows from~\eqref{simpleFunc} and convexity of $x$ and $x'$.
 Indeed, let for definiteness $\int_a^b f_+(t)dt \geq \int_a^b f_-(t)dt$. Then
\begin{gather*}
    \int_a^b f_x(t)dt 
    =
    \left(\int_a^b f_+(t)dt\right)\cdot x + \left(\int_a^b f_-(t)dt\right)\cdot x'
    \\=
    \left(\int_a^b f_-(t)dt\right)\cdot x +
    \left(\int_a^b (f_+(t)- f_-(t))dt\right)\cdot x +
    \left(\int_a^b f_-(t)dt\right)\cdot x'
    \\=
    \left(\int_a^b f(t)dt\right)\cdot x 
    =
    \left(\int_a^b f(t)dt\right)_+\cdot x 
    =
    \left(\int_a^b f(t)dt\right)_+\cdot x + \left(\int_a^b f(t)dt\right)_-\cdot x'.
\end{gather*}
\end{proof}

\begin{lemma}\label{l::derivativeOfSimpleFunc}
Let $X$ be an isotropic $L$-space, $x\in X^{\rm c}\cap X^{\rm inv}$, and $f\colon [a,b]\to \RR$ be a continuously differentiable function. Then  the derivative $\DerH f_x(t)$ exists at each point $t\in [a,b]$ and 
$\DerH f_x(t) = (f'(t))_x.$
\end{lemma}

\begin{proof} First of all note, that if $t,t+\gamma\in [a,b]$, then 
$
f_x(t+\gamma)\DiffH f_x(t) = (f(t+\gamma) - f(t))_x$.
Let $\gamma>0$. If $f'(t)>0$, then for all small enough $\gamma$, $f(t+\gamma) > f(t)$, hence
\begin{gather*}
    \lim\limits_{\gamma\to +0}\gamma^{-1}( f_x(t+\gamma)\DiffH f_x(t))
    =
    \lim\limits_{\gamma\to +0}\gamma^{-1}(f(t+\gamma)- f(t))x
    = f'(t)\cdot x = (f'(t))_x.
\end{gather*}
Analogously in the case $f'(t) <0$.
Finally, if $f'(t) = 0$, then, due to Lemma~\ref{l::inversElementNorm}, 
$$h_X\left(\gamma^{-1}(f_x(t+\gamma)\DiffH f_x(t)) ,\theta\right) =
\left|\gamma^{-1}(f(t+\gamma)- f(t))\right|h_X(x,\theta)
\to 0, \gamma\to+0.
$$
Analogously for the quantity   $\lim\limits_{\gamma\to +0}\gamma^{-1}(f_x(t)\DiffH f_x(t-\gamma))$. The lemma is proved.\end{proof}

\section{On the Korneichuk--Stechkin lemma}\label{s::KSLemma}
\begin{lemma}\label{l::KSLemma}
 Let positive almost everywhere functions $\psi_1\in B([a,a'],\RR)$, $\psi_2\in B([b',b],\RR)$, $a<a'\le b'<b$ such that
$
\int _ {a} ^ {a '} \psi_1 (t) dt =
\int _ {b '} ^ {b} \psi_2 (t) dt
$
be given.
Let also $ \omega$ be a modulus of continuity and the function $ \rho\colon [a, c] \to [c, b],\; c=(a'+b')/2,$ be defined by the relations
\begin{equation}\label{rhoDefinition}
\int_{a}^{s} \psi_1(t) dt =\int_{\rho(s)}^{b} \psi_2(t) dt,\;\text{if}\; s\in [a,a'],\;\;\;\text {and}\;\;\; \rho(s)=a'+b'-s,\; \text{if}\; s\in [a',c].
\end{equation}
Then for an $L$-space $X$ and the functional $S(\psi_1,\psi_2)$ defined in~\eqref{LspacevaluedFunct}
\begin{equation}\label{mainInequality}
    S(\psi_1,\psi_2)
\\ \leq
\int_{a}^{a'} \psi_1(s)\omega(\rho(s)-s)ds = \int_{b}^{b'} \psi_2(t)\omega(t-\rho^{-1}(t))dt.
\end{equation}
If $\omega$ is a concave modulus of continuity and $X$ is isotropic, then~\eqref{mainInequality} turns into equality. 
In this case, the supremum is attained on the functions $(\pm g+\alpha)_x(\cdot) + y\in H^{\omega}([a,b],X)$, where $\alpha\in \RR$, $y\in X$, $x\in X^{\rm c}\cap X^{\rm inv}$, $h_X(x,\theta) = 1$, and
\begin{equation}\label{extremalFunctions}
   g(t) = 
   \begin{cases}
    -\int_t^c\omega'(\rho(s) - s)ds,& a\leq t \leq c,  \\  
    \int_c^t\omega'(s - \rho^{-1}(s))ds,& c\leq t \leq b.
   \end{cases}
\end{equation}  
 \end{lemma}

 \begin{proof} Differentiating equality~\eqref{rhoDefinition}, we get $\psi_1(s) = -\psi_2(\rho(s))\rho'(s)$ for all $s\in [a,a']$. After a substitution  $t=\rho(s)$, we obtain that 
 \begin{equation*}
 \int_{b'}^{b}\psi_2(t)f(t)dt 
 = 
 -\int_{a}^{a'}\psi_2(\rho(s))\rho'(s)f(\rho(s))ds
 =
 \int_{a}^{a'}\psi_1(s)f(\rho(s))ds.   
 \end{equation*}
 Hence
 \begin{gather*}
 h_X\left(\int_{a}^{a'} \psi_1(t)f(t) dt, \int_{b'}^{b}\psi_2(t) f(t) dt\right)
     =
 h_X\left(\int_{a}^{a'} \psi_1(t)f(t) dt,
 \int_{a}^{a'}\psi_1(t)f(\rho(t)) dt\right)
 \\ \leq
 \int_{a}^{a'}h_X\left(\psi_1(t)f(t), \psi_1(t)f(\rho(t)) \right)dt
  \\=
 \int_{a}^{a'}\psi_1(t)h_X(f(t),f(\rho(t)))dt
 \leq
 \int_{a}^{a'}\psi_1(t)\omega(\rho(t) - t)dt
 \end{gather*}
 and the inequality in~\eqref{mainInequality} is proved. The 
 equality in~\eqref{mainInequality} can be obtained after the substitution $s = \rho^{-1}(t)$.
Let now $\omega$ be concave and $X$ be isotropic. For $y\in X$, 
$$\int_{a}^{a'} \psi_1(t)y dt 
= 
\left(\int_{a}^{a'} \psi_1(t) dt\right)P(y)
=
\left(\int_{b}^{b'} \psi_2(t) dt\right)P(y)
=
\int_{b}^{b'} \psi_2(t) ydt,
$$
and  since $X$ is isotropic, we obtain
\begin{gather*}
    h_X\left(\int_{a}^{a'} \psi_1(t)f(t) dt , \int_{b'}^{b}\psi_2(t) f(t) dt\right)
 =
    h_X\left(\int_{a}^{a'} \psi_1(t)f(t) dt +\int_{a}^{a'} \psi_1(t)ydt, \right.
    \\ \left.
    \int_{b'}^{b}\psi_2(t) f(t) dt + \int_{b'}^{b} \psi_2(t)y dt\right) 
    = 
 h_X\left(\int_{a}^{a'} \psi_1(t)(f(t) +y) dt, \int_{b'}^{b}\psi_2(t) (f(t)+y) dt\right)
\end{gather*}
and hence if the supremum in~\eqref{mainInequality} is attained on some function $f$, then it is attained on all functions $f(\cdot)+y$, $y\in X$. 
Let $G(\cdot) = g(\cdot) + \alpha$, where $g$ is defined in~\eqref{extremalFunctions} and $\alpha\in \RR$. It is known (see~\cite[\S~7.1]{ExactConstants}), that the function $G$ belongs to $H^\omega([a,b],\RR)$ and is extremal in~\eqref{mainInequality} for real-valued functions. 
 Let $I=\int_a^{a'} \psi_1(t)G(t) dt$ and $J=\int_{b'}^b\psi_2(t) G(t) dt$.  By~\eqref{extremalFunctions}, the function $g$ is non-decreasing. Hence there are three possibilities 1) $I\geq 0, \; J\geq 0$, 2)~$I\leq 0, \; J\geq 0$ and 3) $I\leq 0, \; J\leq 0$. Considering each of them and taking into account Lemmas~\ref{l::distBetweenConvexElemMultipliers}, \ref{l::inversElementNorm} and~\ref{l::LspaceValuedFunc} we obtain
\begin{gather*}
 h_X\left(\int_a^{a'} \psi_1(t)G_x(t) dt, \int_{b'}^b\psi_2(t) G_x(t) dt\right) 
 = |I-J|h_X(x,\theta) = \int_a^{a'}\psi_1(t)\omega(\rho(t) - t)dt.
 \end{gather*}
 We elaborate the proof in the case $I \leq 0$ and $J \geq 0$:
 $$
 h_X\left(\int_a^{a'} \psi_1(t)G_x(t) dt, \int_{b'}^b\psi_2(t) G_x(t) dt\right)=h_X(I_-\cdot x',J_+\cdot x)
 $$
 $$
 =h_X(I_-\cdot x'+I_-\cdot x, I_-\cdot x+J_+\cdot x)=h_X(\theta, (-I+J)\cdot x)=|I-J|h(x,\theta).
 $$
 The case $G(\cdot)=-g(\cdot)+\alpha$, can be considered analogously.
The lemma is proved. \end{proof}

Recall, that for a measurable non-negative function $f\colon [a,b]\to\RR$ the function
$$ 
m(f,y):={\rm mes}\{ t\in [a,b]\colon f(t)>y\},\;y\in\RR
$$
is called the distribution function of $f$. The function 
$$
r(f,t):=\inf\{ y\colon m(f,y)\le t\},\;t\in [0,b-a]
$$
is called the non-increasing (or Hardy's) rearrangement of $f$. $r(f,\cdot)$ is a non-increasing on $[0,b-a]$, and equimeasurable with $f$ function. 
\begin{remark}
For a concave modulus of continuity $\omega$, and an isotropic $L$-space $X$,  the statement of Lemma~\ref{l::KSLemma} can be rewritten as follows:
\begin{equation}\label{rearrangement}
S(\psi_1,\psi_2)=     \left|\int_{0}^{b-a}r' (\Psi,s)\omega(s)ds\right| = \int_{0}^{b-a}r (\Psi,s)\omega'(s)ds,
\end{equation}
where
 $   \Psi (s)=\int_a^s(\psi_1(u)-\psi_2(u))du,\;s\in[a,b].$
\end{remark}
The equality of the quantities in the right-hand sides of inequalities~\eqref{mainInequality} and~\eqref{rearrangement} for concave $\omega$ was proved by Korneichuk, see e.~g.~\cite[Lemma~7.1.2]{ExactConstants}.

\section{Estimate for the functional $S(\psi_1,\psi_2)$ and Ostrowski type inequalities}
\subsection{General estimate for the functional $S(\psi_1,\psi_2)$}
\begin{definition}
A function $\varphi\in C([a,b],\RR)$
is called a hat-function, if
\begin{enumerate}
    \item $\varphi(a) =\varphi(b)=0$, $|\varphi(t) |>0$ for $a<t<b$, and
    \item $\forall y\in (0,\max\limits_{t\in[a,b]}|\varphi(t)|)$ the equation $|\varphi(t)|=y$ has exactly two roots on $(a,b)$.
\end{enumerate}
Each hat-function $\varphi$ is continued to the whole line by $\varphi(t) = 0$, $t\notin (a,b)$.
\end{definition}

Denote by  $D[a,b]$ the set of functions $\psi\colon [a,b]\to\RR$ that have finite one-sided limits $\psi(t+0)$ and $\psi(t-0)$ for all $t\in (a,b)$, and finite limits $\psi(a+0)$ and $\psi(b-0)$. Set
$
D_0[a,b]=\{ \psi\in D[a,b]\colon \int_a^b\psi(t)dt=0 \}$ and
$
D^1_0[a,b]=\{ f(t)=\int_a^t\psi(u)du\colon \psi \in D_0[a,b] \}.
$
As it is known (see e.~g.~\cite{Korneichuk71} and~\cite[Chapter~7]{ExactConstants}), each function $f\in D^1_0[a,b]$ can be represented as a finite or countable sum of hat-functions
\begin{equation}\label{sigma}
f(t)=\sum\nolimits_k\varphi_k(t).
\end{equation}
This equality is called the $\Sigma$-representation of $f$. The following properties are satisfied (see~\cite[Chapter~7]{ExactConstants}).
1) $|f(t)|=\sum\nolimits_k|\varphi_k(t)|,\; a\le t\le b;$
  2) The intervals $(\alpha_k,\alpha_k')$, $(\beta_k',\beta_k)$, on which the functions $\varphi_k(t)$ are strictly monotone, are pairwise disjoint, and on each of them $\varphi_k(t)=f(t)+c_k$, $c_k\in\RR$, and hence $f'(t)=\sum\nolimits_k\varphi_k'(t)$ almost everywhere on $[a,b]$;
  3) $\int_a^b|f(t)|dt=\sum\nolimits_k\int_{\alpha_k}^{\beta_k}|\varphi_k(t)|dt$;
    4) $ \bigvee_a^b f =\sum\nolimits_k\bigvee_{\alpha_k}^{\beta_k}(\varphi_k).$

\begin{definition}
For a function $f\in D^1_0[a,b]$ with $\Sigma$-representation~\eqref{sigma}, the Korneichuk $\Sigma$-rearrangment of $f$ is defined by equality
$$
R(f;t)=\sum\nolimits_kr(|\varphi_k|,t), \; 0\le t\le b-a.
$$
\end{definition}

\begin{theorem}\label{th::KSLemmaGeneralEstimate}
Let $\omega$ be a concave modulus of continuity and $\psi_1, \psi_2 \in B([a,b],\RR)$ be such that $\psi_1- \psi_2\in D_0[a,b]$. Set $\Psi(t)=\int_a^t[\psi_1(u)- \psi_2(u)]du$. Then
\begin{equation}\label{estS}
    S(\psi_1, \psi_2)\le \int_0^{b-a}|R'(\Psi;t)|\omega(t)dt.
\end{equation}
\end{theorem}

\begin{proof} Set $E_\pm =\{ t\in [a,b]\colon \pm \psi_1(t)\ge \pm \psi_2(t)\}$. If $P$ is the convexifying operator (see Definition~\ref{def::convexifyingOperator}), then $\psi_1(s)P(f(s)) = (\psi_1(s)-\psi_2(s))P(f(s)) + \psi_2(s)P(f(s))$ for  all $s\in E_+$ and $\psi_2(s)P(f(s)) = (\psi_2(s)-\psi_1(s))P(f(s))+ \psi_1(s)P(f(s))$ for  all $s\in E_-$. Hence for any function
$f\in H^\omega([a,b],X)$ we obtain
\begin{gather*}
h_X\left( \int_a^bf(t)\psi_1(t)dt,\int_a^bf(t)\psi_2(t)dt\right)
\\
=
h_X\left( \int_{E_+}f(t)(\psi_1(t)-\psi_2(t))dt+\int_{E_-}f(t)\psi_1(t)dt +\int_{E_+}f(t)\psi_2(t)dt,\right.
\\
\left.\int_{E_-}f(t)(\psi_2(t)-\psi_1(t))dt
+\int_{E_-}f(t)\psi_1(t)dt 
+\int_{E_+}f(t)\psi_2(t)dt\right)
\\
\leq
h_X\left( \int_{E_+}f(t)(\psi_1(t)-\psi_2(t))dt,\int_{E_-}f(t)(\psi_2(t)-\psi_1(t))dt\right)
\\
=h_X\left( \int_{a}^bf(t)(\psi_1(t)-\psi_2(t))_+dt,\int_{a}^bf(t)(\psi_1(t)-\psi_2(t))_-dt\right).
\end{gather*}
Moreover, the inequality in the above chain becomes equality in the case of isotropic space $X$. Let $\Psi(t)=\sum\nolimits_k\varphi_k(t)$ be the  $\Sigma$-representation of the function $\Psi.$ Since $\Psi'=\psi_1-\psi_2=\sum\nolimits_k\varphi_k'$, due to the mentioned above properties of $\Sigma$-representations, 
$$
(\psi_1-\psi_2)_+=\sum\nolimits_k(\varphi_k')_+,\;\;\; (\psi_1-\psi_2)_-=\sum\nolimits_k(\varphi_k')_-.
$$
Moreover, the functions $(\varphi_k')_+$ and $(\varphi_k')_-$ satisfy the conditions of Lemma~\ref{l::KSLemma} for each $k$. Applying Lemma~\ref{l::KSLemma} (more precisely, equality~\eqref{rearrangement}), we obtain
\begin{gather*}
h_X\left( \int_{a}^bf(t)(\psi_1(t)-\psi_2(t))_+dt,\int_{a}^bf(t)(\psi_1(t)-\psi_2(t))_-dt\right)
\\
=h_X\left( \int_{a}^bf(t)\sum\nolimits_k(\varphi_k')_+dt,\int_{a}^bf(t)\sum\nolimits_k(\varphi_k')_-dt\right)\\
\le \sum\nolimits_kh_X\left( \int_{a}^bf(t)(\varphi_k')_+dt,
\int_{a}^bf(t)(\varphi_k')_-dt\right)
\\
\le \sum\nolimits_k\int_{0}^{b-a}r(|\varphi_k|,t)\omega'(t)dt=\int_{0}^{b-a}R(\Psi,t)\omega'(t)dt. 
\end{gather*}
The theorem is proved.\end{proof}

\begin{remark}\label{r::Sharpness}
In the case of isotropic $X$, estimate~\eqref{estS} is sharp, provided the exremal in Lemma~\ref{l::KSLemma} functions for  $\psi_1=(\varphi_k')_+$ and $\psi_2=(\varphi_k')_-$ can be ''glued'' so that the obtained function belongs to $H^\omega([a,b],X)$. 

\end{remark}
Assume that the $\Sigma$-representation of the function $\Psi$ is
$
\Psi(t)=\sum\nolimits_{k=1}^n\varphi_k(t),
$
if $[\alpha_k,\beta_k]$ is the support of the hat-function $\varphi_k$, $k=1,\ldots,n$, then
$$
\alpha_1<\beta_1\le\alpha_2<\beta_2\le\ldots\le \alpha_n<\beta_n,
$$
and on the segments $[\alpha_k,\beta_k]$ and $[\alpha_{k+1},\beta_{k+1}]$ the functions $\varphi_k$ and $\varphi_{k+1}$ have opposite signs, $k=1,\ldots, n-1$.
Below we sketch the procedure of gluing. We start with the case $X = \RR$. Let $g_k\in H^\omega([\alpha_k,\beta_k],\RR)$ be the extremal for the functional $S((\varphi_k')_+,(\varphi_k')_-)$ function. On the set $\bigcup_{k=1}^n[\alpha_k,\beta_k]$ define the function $g$, setting $g(t) = g_k(t) +c_k$, $t\in [\alpha_k,\beta_k]$, where $c_k$ are such that $g(\beta_k) = g(\alpha_{k+1})$, $k=1,\ldots, n-1$. Next, we continue $g$ to the whole segment $[a,b]$ setting $g(t) = g(\alpha_1)$, if $t\leq \alpha_1$, $g(t) =g(\beta_k)$, if $t\in (\beta_k,\alpha_{k+1}),\; k=1,\ldots, n-1,$ and $g(t) =g(\beta_k)$, if $t\geq \beta_n$.
Lemma~4.1 from~\cite{stepanets2018} contains a criteria for $g$ to belong to $H^\omega([a,b],\RR)$. In particular, this is true, if 
for some $m$,  $1\le m\le n$, 
$$
\beta_1-\alpha_1\le\beta_2-\alpha_2\le\ldots\le\beta_m-\alpha_m \text{ and } \beta_m-\alpha_m\ge\beta_{m+1}-\alpha_{m+1}\ge\ldots\ge \beta_n-\alpha_n.
$$
If $g\in H^\omega([a,b],\RR)$, then the function $(g+\gamma)_x(\cdot) + y$ with $\gamma\in \RR$, $y\in X$ and $x\in X^{\rm c}\cap X^{\rm inv}$, $h_X(x,\theta) = 1$, is a glued extremal for~\eqref{estS}.

\subsection{Ostrowski type inequalities.}
The following theorem is a wide generalization of Theorem~2 from~\cite{barnet}, in which $X = \RR$, $\omega(t) = t$, and $[c,d]\subset [a,b]$.

\begin{theorem}\label{th::ostrowskiInequality}
Let two segments $[a,b]$ and $[c,d]$ be given. Set $M = \max\{b-a,d-c\}$, $m = \min\{b-a,d-c\}$, for $\alpha,\beta\geq 0$ set $I(\alpha,\beta) = \int_{\alpha}^{\beta}\omega(s)ds$ and assume for definiteness that $a\leq c$. Then for all $f\in H^\omega([a,\max\{b,d\}],X)$
\begin{multline*}
h_X\left(\frac{1}{b-a}\int_a^b{f(t)dt}, \frac {1}{d-c}\int_c^d{f(t)dt}\right)
\\ \leq
 \begin{cases}
\frac{M-m}{M^2}\left\{I\left(0,\frac{M(c-a)}{M-m}\right) + I\left(0,\frac{M(b-d)}{M-m}\right)\right\},& [c,d]\subset [a,b],\\
\frac{1}{M+m}I\left(\frac{M(b-c)}{m}, d-a\right) + 
\frac{M-m}{M^2}I\left(0,\frac{M(b-c)}{m}\right), & b\in [c,d],\\
\frac{1}{M+m}I(c-b, d-a), & c\geq b.
 \end{cases}
\end{multline*}
If $X$ is isotropic and $\omega$ is concave, then the inequality is sharp.
\end{theorem}
The theorem follows from Theorem~\ref{th::KSLemmaGeneralEstimate}  applied to the functions $\psi_1=\frac 1{b-a}\chi_{[a,b]}$ and $\psi_2=\frac 1{d-c}\chi_{[c,d]}$. We omit the technical details of the proof.  The extremal function can be obtained using the described above procedure of gluing.


Direct computations show that Theorem~\ref{th::ostrowskiInequality} implies the following result.
\begin{corollary}\label{c::symmetricalInterval}
If in Theorem~\ref{th::ostrowskiInequality} additionally $c+d = a+b$, i.~e. the midpoints of the intervals $(a,b)$ and $(c,d)$ coincide, then
$$ 
h_X\left(\int_a^b{f(t)dt}, \frac {b-a}{d-c}\int_c^d{f(t)dt}\right)
 \leq
 \frac {4(c-a)}{b-a}\int_0^{(b-a)/2}{\omega(t)dt}.
$$
If $X$ is isotropic and $\omega$ is concave, then the inequality is sharp.
\end{corollary}

    Applying Theorem~\ref{th::ostrowskiInequality} to the segment that contains $t$ and the segment $[c,d]$ (while both are contained in $[a,b]$) and then shrinking the first one into a point, we obtain

\begin{corollary}\label{c::valueIntegralDeviation}
Let $t\in [a,b]$, $[c,d]\subset [a,b]$ and $\omega (\cdot)$ be an arbitrary modulus of continuity. If $f\in H^\omega([a,b], X)$ and $P$ is the convexifying operator, then
\begin{equation}\label{ostr}
h_X\left( P(f(t)),\frac 1{d-c}\int_c^d f(u)du\right)
\leq \frac{1}{d-c}\int_c^d\omega(|s-t|)ds.
\end{equation}
If $X$ is isotropic, then the inequality is sharp. An extremal function is  $(\omega(|\cdot -t|))_x$, where $x\in X^{\rm c}, h_X(x,\theta)=1$.
\end{corollary}
 Let a segment $[a,b]$ and numbers $t,h$ such that $a\le t<t+h\le (a+b)/2$ be given. Applying Theorem~\ref{th::KSLemmaGeneralEstimate} to $\psi_1=\frac 1{b-a}\chi_{[a,b]}$ and $\psi_2=\frac 1{2h}\left(\chi_{[t,t+h]}+\chi_{[a+b-t-h,a+b-t]}\right)$, and passing to the limit as $h\to 0$, we obtain a generalization of \cite[Theorem~2]{Guessab02}.
\begin{corollary}
For arbitrary $f\in H^\omega([a,b], X)$ and $t\in [a,(a+b)/2)$
\begin{multline}\label{guessab}
    h_X\left(\frac 12 (P(f(t))+P(f(a+b-t))),\frac 1{b-a}\int_a^bf(u)du\right)
       \\ \le \frac 2{b-a}\left(\int_0^{t-a}\omega(u)du+\int_0^{({a+b-2t})/2}\omega(u)du\right). 
\end{multline}
Inequality~\eqref{guessab} becomes equality for $f(u)=\min\{\omega(|u-t|),\omega(|u+t-a-b|)\}\cdot x$, $x\in X^{\rm c}$.
\end{corollary}

\begin{remark}
Inequalities~\eqref{ostr} and~\eqref{guessab} can easily be proved directly.
\end{remark}

\section{On optimal recovery problems on the class $H^\omega([a,b],X)$}\label{s::HOmegaRecovery}
In this section we consider the problems of recovery of the convexifying operator $P$ and the integral 
$
\Lambda(f) = \int_a^b f(t) dt
$ 
on the class $H^\omega([a,b],X)$, given the information operator 
$
I_{\bf t}(f) 
= 
\left(\frac 1 {2h}\int_{t_1 - h}^{t_1 + h}f(t)dt,\dots,\frac 1 {2h}\int_{t_n - h}^{t_n + h}f(t)dt\right),
$
where $n\in\NN$, $h > 0$ and ${\bf t} := (t_1,\dots,t_n)$ are such that  
\begin{equation}\label{knots}
a\leq t_1-h<t_1+h<t_2-h<\ldots<t_n +h \leq b,
\end{equation}
using arbitrary method of recovery  $\Phi\colon X^n\to B([a,b],X)$ and $\Phi\colon X^n\to X$ respectively.
Define a vector $\tau = \tau({\bf t})$ with components
\begin{equation}\label{tau}
\tau_1 = a, \tau_i = \frac 12(t_{i-1} + t_{i}), i = 2,\ldots, n, \tau_{n+1} = b
\end{equation}
and set
\begin{equation}\label{t-star}
{\bf t}^* = (t_1^*,t_2^*\ldots, t_n^*) = \left(\frac{b-a}{2n}, \frac{3(b-a)}{2n},\ldots, \frac{(2n-1)(b-a)}{2n}\right).
\end{equation}
  We need the following well known estimate for the value of the optimal recovery~\eqref{errorOfRecovery}.
 \begin{lemma}\label{l::errorOfRecoveryFromBelow}
If $f,g\in W$ are such that $I(f) = I(g)$, then 
 $$ {\mathcal E}(\Lambda,W, I, X)\geq \frac 12 h_X(\Lambda(f),\Lambda(g)).$$
\end{lemma}
\begin{proof} We have
\begin{gather*}
    \sup_{z\in W} h_X(\Lambda(z), \Phi(I(z)))
    \geq
    \max\left\{h_X(\Lambda(f), \Phi(I(f))), h_X(\Lambda(g), \Phi(I(g)))\right\}
   \\ \geq 
     \frac12\left(h_X(\Lambda(f), \Phi(I(f))) + h_X(\Lambda(g), \Phi(I(f)))\right)
\geq 
      \frac12 h_X(\Lambda(f), \Lambda(g)).
\end{gather*}
\end{proof}
\subsection{Real-valued extremal functions}
For given $n\in\NN$, $h>0$ and ${\bf t}$ that satisfy~\eqref{knots}, denote by $H_{\bf t}^h$ the class of functions $y\in H^\omega([a,b],\RR)$ such that $\int_{t_i-h}^{t_i+h}y(t)dt = 0$ for all $i = 1,\ldots, n$. Note, that for arbitrary $f\in H_{\bf t}^h$,  $x\in X^{\rm c}\cap X^{\rm inv}$, due to Lemma~\ref{l::LspaceValuedFunc}, $\int_{t_i-h}^{t_i+h}f_x(t)dt =\theta$, $i = 1,\ldots, n$. 
\begin{lemma}\label{l::IdRecoveryLowerBound}
Let numbers $n\in\NN$, $h>0$ and ${\bf t} := (t_1,\dots,t_n)$ that satisfy~\eqref{knots} be given. For arbitrary modulus of continuity $\omega$ there exists a function $f_{\bf t}\in H_{\bf t}^h$ such that 
$$
\max\limits_{t\in[a,b]}|f_{\bf t}(t)|
\geq 
\frac 1{2h}\int_{({b-a})/({2n})-h}^{({b-a})/({2n})+h}\omega(u)du.
$$
\end{lemma}
\begin{proof}
 Among $2n$ segments $[\tau_i,t_i]$ and $[t_i,\tau_{i+1}]$, $i = 1,\ldots, n$, there exists at least one with length at least $\frac{b-a}{2n}$. Let for definitness it be the segment $[\tau_{i^*},t_{i^*}]$, $i^*\in \{1,\ldots, n\}$. We define a functions $f_{\bf t}$ on the segment $[\tau_{i^*},t_{i^*}+h] = [a,t_{1}+h]$, if  $i^* = 1$, or on the segment $[t_{i^*-1}-h,t_{i^*}+h]$, if  $i^* > 1$, by the formula
$$
f_{\bf t}(u)
=
\frac 1{2h}\int_{t_{i^*}-h}^{t_{i^*}+h}\omega(|s-\tau_{i^*}|)ds-\omega (|u-\tau_{i^*}|).
$$
Next we continue this function to the whole segment $[a,b]$ as follows. We set $f_{\bf t}(u) = f_{\bf t}(t_{i^*}+h)$ on $[t_{i^*}+h,t_{i^*+1}-h]$; $f_{\bf t}(u) =  f_{\bf t}(t_{i^*}+t_{i^*+1}-u)$ on $[t_{i^*+1}-h,t_{i^*+1}+h]$; $f_{\bf t}(u) = f_{\bf t}(t_{i^*+1}+h)$ on $[t_{i^*+1}+h,t_{i^*+2}-h]$; $f_{\bf t}(u) =  f_{\bf t}(t_{i^*+1}+t_{i^*+2}-u)$ on $[t_{i^*+2}-h,t_{i^*+2}+h]$ and so on. The process goes analogously for $u<t_{i^*}-h$.

From the definition it follows that $f_{\bf t}\in H^h_{\bf t}$ and
$$
\max\limits_{t\in[a,b]}|f_{\bf t}(t)|
\geq 
f_{\bf t}(\tau_{i^*})
=
\frac 1{2h}\int_{t_{i^*}-h}^{t_{i^*}+h}\omega(|u-\tau_{i^*}|)du
\geq 
\frac 1{2h}\int_{({b-a})/({2n})-h}^{({b-a})/({2n})+h}\omega(u)du.
$$
\end{proof}

\begin{lemma}\label{l::lowerBound}
Let numbers $n\in\NN$, $h>0$ and ${\bf t} := (t_1,\dots,t_n)$ that satisfy~\eqref{knots} be given. Let $\omega$ be a concave modulus of continuity. Then there exists a function $f_{\bf t}\in H_{\bf t}^h$ such that 
$$\int_a^bf_{\bf t}(t)dt\geq  2n\left(1-\frac{2nh}{b-a}\right)\int_0^{(b-a)/ (2n)}\omega(t)dt.$$
\end{lemma}
\begin{proof}
Consider the even function $y_0$, defined on $[0,\infty)$ by the following equation.
\begin{equation}\label{lowerBound.1}
y_0(t) = 
\begin{cases}
		-\frac{2nh}{b-a}\omega\left(\frac {b-a} {2nh}(h-t)\right), & t\in [0,h], \\  
        \frac{b-a-2nh}{b-a}\omega\left(\frac {b-a}{{b-a}-2nh}(t-h)\right), & t\in \left[h, \frac {b-a} {2n}\right], \\ 
        y_0\left(\frac {b-a} {2n}\right), & t > \frac {b-a} {2n}.
\end{cases}
\end{equation}
Note, that the restriction of the function $y_0$ to the segment $\left[0,({b-a})/({2n})\right]$ is the function built according to formula~\eqref{extremalFunctions} with $\psi_1 =\frac 1h\chi_{[0,h]}$  and $\psi_2 = \frac {2n}{b-a-2nh} \chi_{\left[h,({b-a})/({2n})\right]}$.
Hence $y_0(t)\in H^\omega([0,(b-a)/(2n)],\RR)$, since $\omega$ is a concave modulus of continuity. 
Set $$y_1(t) := \min\{y_0(t-t_1), y_0(t-t_2),\dots,y_0(t-t_n)\},\; t\in\RR.$$ 

Set $s_0:=a$, $s_i = (t_i+t_{i+1})/2$, $i=1,\ldots, n-1$, $s_n:=b$.
Then  $y_1(t) = y_0(t-t_k)$, $t\in [s_{k-1}, s_k]$, $k=1,\dots,n$. Note, that  $y_1(t)\in H^\omega([a,b],\RR)$. Set $y(t):=y_1(t) + C$, where the constant $C$ is chosen in such a way that $\int_{-h}^h (y_0(t)+C)dt=0$. This implies 
\begin{equation}\label{lowerBound.2}
\int_{t_k-h}^{t_k+h} y(t)dt=0,\;k=1,\dots,n.
\end{equation}
Hence $y\in H_{\bf t}^h$. We estimate the integral $\int_a^by(t)dt$ from below. The function $y_0$ is even and  $J(t) := \int_{0}^{t}y_0(s)ds$ is convex, since $y_0$ is non-decreasing on $[0,\infty)$. Hence
\begin{gather*}
\int_{a}^{b} y(t)dt=C(b-a)+\int_{a}^{b} y_1(t)dt=C(b-a)+\sum\nolimits_{k=1}^n\int_{s_{k-1}}^{s_k}y_0(t-t_k)dt=
C(b-a)
\\+\sum\nolimits_{k=1}^n\int_{s_{k-1}-t_k}^{s_k-t_k}y_0(t)dt
=
 C(b-a)+\sum\nolimits_{k=1}^n J(s_k-t_k) +\sum\nolimits_{k=1}^n J(t_k - s_{k-1}) \geq C(b-a)
 \\
 +
 2n J\left(\frac 1 {2n}\sum\nolimits_{k=1}^n (s_k-t_k) + \frac 1 {2n}\sum\nolimits_{k=1}^n (t_k - s_{k-1})\right)=C(b-a)+2nJ\left(\frac {b-a} {2n}\right).
\end{gather*}
Using~\eqref{lowerBound.1} and~\eqref{lowerBound.2} to compute the right-hand side of the latter inequality, we obtain
$$C(b-a)+2nJ\left(\frac {b-a} {2n}\right) = 2n\left(1-\frac{2nh}{b-a}\right)\int_0^{(b-a)/ (2n)}\omega(t)dt$$
and the lemma is proved.\end{proof}
\subsection{Optimal recovery of the  convexifying operator.}
\begin{theorem}
Let numbers $n\in\NN$, $h>0$ and ${\bf t} := (t_1,\dots,t_n)$ that satisfy~\eqref{knots} be given. For the convexifying operator $P$ and arbitrary modulus of continuity $\omega$
$$
 \inf\limits_{\bf t}{\mathcal E}(P,H^\omega([a,b],X), I_{\bf t}, B([a,b],X))
 =
 \frac 1{2h}\int_{({b-a})/({2n})-h}^{({b-a})/({2n})+h}\omega(u)du.
 $$
The optimal informational operator is $I_{{\bf t}^*}$ and the optimal recovery method is 
$$\Phi^*(I_{{\bf t}^*}(f))(u) = \frac{1}{2h}\int_{t_k^* - h}^{t_k^* + h}f(t)dt,\;u\in [\tau_{k}({\bf t}^*),\tau_{k+1}({\bf t}^*)],$$
where the vectors ${\bf t}^*$ and $\tau = \tau({\bf t}^*)$ are defined in~\eqref{t-star} and~\eqref{tau} respectively.
\end{theorem}
\begin{proof}
For $f\in H^\omega([a,b],X)$, $t\in [\tau_i({\bf t}^*),\tau_{i+1}({\bf t}^*)]$, and $i\in\{1,\ldots, n\}$,  Corollary~\ref{c::valueIntegralDeviation} implies
\begin{multline*}
h_X(P(f(t)),\Phi^*(I_{{\bf t}^*})(t))
=
h_X\left(P(f(t)),\frac 1{2h}\int_{t_i^*-h}^{t_i^*+h}f(u)du\right)
\le
\frac 1{2h}\int_{t_i^*-h}^{t_i^*+h}\omega(|u-t|)du
\\ \le
\frac 1{2h}\int_{t_i^*-h}^{t_i^*+h}\omega(|u-\tau_i({\bf t}^*)|)du
=
\frac 1{2h}\int_{({b-a})/({2n})-h}^{({b-a})/({2n})+h}\omega(u)du.
\end{multline*}
Hence
$$
{\mathcal E}(P,H^\omega([a,b],X), I_{{\bf t}^*}, B([a,b],X))
\le 
\frac 1{2h}\int_{({b-a})/{(2n)}-h}^{({b-a})/({2n})+h}\omega(u)du.
$$
Choose $x\in X^{\rm c}\cap X^{\rm {inv}}$, $h_X(x,\theta)=1$, and for the function $f_{\bf t}$ from Lemma~\ref{l::IdRecoveryLowerBound} set
\begin{equation}\label{LValuedExtremalFunc}
    \underline{F}_n=(f_{\bf t})_{x'} \text{ and }
\overline{F}_n=(f_{\bf t})_x.
\end{equation}
Note, that the functions $\underline{F}_n$ and $\overline{F}_n$ are convex-valued. Using Lemma~\ref{l::errorOfRecoveryFromBelow}, we obtain
\begin{multline*}
{\mathcal E}(P,H^\omega([a,b],X), I_{\bf t}, B([a,b],X))
\ge 
\frac 12\max\limits_{t\in[a,b]}h_X(\overline{F}_n(t),\underline{F}_n(t))
\\ =
\max\limits_{t\in[a,b]}|f_{\bf t}(t)|
\geq 
\frac 1{2h}\int_{({b-a})/({2n})-h}^{({b-a})/({2n})+h}\omega(u)du.
\end{multline*}
The theorem is proved. \end{proof}

\subsection{Optimal recovery of the integral}
\begin{theorem}
Let numbers $n\in\NN$, $h>0$, and ${\bf t} := (t_1,\dots,t_n)$ that satisfy~\eqref{knots} be given.  For a concave modulus of continuity $\omega$,
$$
 \inf\limits_{\bf t}{\mathcal E}(\Lambda,H^\omega([a,b],X), I_{\bf t}, X)
 =
 2n\left(1-\frac{2nh}{b-a}\right)\int_0^{(b-a)/ (2n)}\omega(t)dt.
 $$
The optimal informational operator is $I_{{\bf t}^*}$ and the optimal recovery method is 
$$\Phi^*(I_{{\bf t}^*}(f)) = \frac {b-a}{n}\sum\nolimits_{k=1}^n\frac{1}{2h}\int_{t_k^* - h}^{t_k^* + h}f(t)dt,$$
where the vector ${\bf t}^*$  is defined in~\eqref{t-star}.
\end{theorem}
\begin{proof}
For each $f\in H^\omega([a,b],X)$ we obtain, using Corollary~\ref{c::symmetricalInterval},
\begin{multline*}
h_X\left(
\int_{a}^bf(t)dt, \Phi^*(I_{{\bf t}^*}(f))\right)
\leq
\sum_{k=1}^nh_X\left(
{\int_{ \frac{(k-1)(b-a)}{n}}^{ \frac{k(b-a)}{n}}f(t)dt},
\frac {b-a}{2nh}
{\int_{ \frac{(2k-1)(b-a)}{2n}- h}^{ \frac{(2k-1)(b-a)}{2n} + h}f(t)dt}\right)
\\ \leq 
\sum_{k=1}^n 2\left(1-\frac{2nh}{b-a}\right)
\int_0^{(b-a)/ (2n)}\omega(t)dt
=
2n\left(1-\frac{2nh}{b-a}\right)\int_0^{(b-a)/(2n)}\omega(t)dt.
\end{multline*}
Hence
$
 \inf\limits_{\bf t}{\mathcal E}(\Lambda,H^\omega([a,b],X), I_{\bf t}, X)
\leq 
2n\left(1-\frac{2nh}{b-a}\right)\int_0^{(b-a)/ (2n)}\omega(t)dt.
$
Now let an information set ${\bf t}$ be fixed and $f_{\bf t}$ be built according to Lemma~\ref{l::lowerBound}. Using functions~\eqref{LValuedExtremalFunc} with such $f_{\bf t}$,  Lemmas~\ref{l::errorOfRecoveryFromBelow} and~\ref{l::lowerBound}, we obtain
\begin{gather*}
{\mathcal E}(\Lambda,H^\omega([a,b],X), I_{\bf t}, X)
\\ \geq
\frac 12h_X\left(\int_a^b \overline{F}_n(t)dt , \int_a^b \underline{F}_n(t)dt\right) 
 =
\int_a^b f_{\bf t}(t)dt
\geq 
2n\left(1 - \frac{2nh}{b-a}\right)\int_0^{ (b-a)/(2n)}\omega(t)dt.
\end{gather*}
\end{proof}

\section{On optimal recovery problems on the class $W^1H^\omega ([a,b],X)$}\label{s::PolylineApproximation}
\begin{definition}\label{W1Homega}
Given a modulus of continuity $\omega(\cdot)$, denote by $W^1H^\omega([a,b],X)$ the class of functions $f$ of the form $$f(t)=x+\int_a^t\phi(s)ds, \text{ where } \phi\in H^\omega([a,b],X),\, x\in X^c.$$
\end{definition}
Note, that such a function $f$ is convex-valued, has Hukuhara type derivative $\DerH f$ and $\DerH f=P(\phi)$, see \cite[Lemma~2.24]{babenko19}, where $P$ is the convexifying operator.
\subsection{Real-valued extremal functions}
Let a partition ${\bf t} = (t_0,\ldots, t_n)$ of the segment $[a,b]$ be given, 
\begin{equation}\label{partition}
a = t_0<t_1<\ldots <t_{n-1}<t_n = b.
\end{equation}

\begin{lemma}\label{l::omegaSpline}
Let $\omega$ be a concave modulus of continuity and a partition ${\bf t}$ be given. Then there exists a function $f_{\bf t}\in W^1H^\omega ([a,b],\RR)$ such that $f_{\bf t}(t_i) = 0$, $i = 0,1,\ldots, n$, and  
\begin{equation}\label{splineNorm}
    \max\limits_{t\in [a,b]}|f_{\bf t}(t)|\geq \frac 14\int_0^{(b-a)/n}\omega(u)du.
\end{equation}
If the partition ${\bf t}$ is uniform, then inequality~\eqref{splineNorm} becomes equality.
\end{lemma}
\begin{proof}
In the space $\RR^{n+1}$ consider the sphere
 $$
 \mathbb{S}^n=\left\{ \xi=(\xi_1,\ldots ,\xi_{n+1})\in \RR^{n+1}\colon \sum\nolimits_{i=1}^{n+1}|\xi_i|=b-a\right\}.
 $$
  Each $\xi\in \mathbb{S}^n$ generates a set of points on the segment $[a,b]$
  $$
  \eta_0(\xi) = a,\;\eta_1(\xi)=a+|\xi_1|,\; \eta_2(\xi)=\eta_1+|\xi_2|,\ldots , \eta_n(\xi)=\eta_{n-1}+|\xi_n|,\; \eta_{n+1}(\xi) = b.
  $$
  Let $h(t) = \frac 12 \omega(2|t|)$, $t\in \RR$. For $\xi\in \mathbb{S}^n$, set $h_\xi(t) = \min\limits_{k = \overline{1,n}} h(t-\eta_k(\xi))$ and   
  $g_\xi(t) = h_\xi(t)\cdot {\rm sgn\,}\xi_i$ for $t\in [\eta_{i-1}(\xi),\eta_{i}(\xi)]$, $i = 1, \ldots, n+1$. Then, due to concavity of $\omega$, $g_\xi\in H^\omega([a,b],\RR)$.
  Set
  $$
   G_\xi(t)=\int_a^tg_\xi(u)du
  $$
  and define the vector field on $\mathbb{S}^n$, by the formula
  $
  \xi \to (G_\xi(t_1),\ldots ,G_\xi(t_n)).
  $
  It is easy to see, that this field is continuous and odd. The Borsuk theorem implies that there exists $\xi^*=\xi^*({\bf t})=(\xi^*_1,\ldots ,\xi^*_{n+1})\in \mathbb{S}^n$ such that 
  $
  G_{\xi^*}(t_1)=G_{\xi^*}(t_2)=\ldots =G_{\xi^*}(t_n)=0.
  $
  Moreover, $G_{\xi^*}(a)=0$. Hence the function $ G_{\xi^*}$ has at least  $n+1$ zeros on $[a,b]$ and thus $g_{\xi^*}=G'_{\xi^*}$ has at least $n$ changes of sign. Since $g_{\xi^*}$ can change its sign only at the points $\eta_1(\xi^*),\ldots ,\eta_n(\xi^*)$, all these points are distinct, $g_{\xi^*}$ has exactly $n$ sign changes on $[a,b]$, and $\eta_i(\xi^*)$ is the unique point of local extremum of $G_{\xi^*}$ inside the segment $[t_{i-1},t_i]$, $i=1,\ldots, n$.
    Since $\omega$ is non-decreasing, the function $u\to\int_0^{u}\omega(t)dt$ is convex, hence applying the Jensen inequality we obtain
  \begin{gather*}
  \bigvee\nolimits_a^bG_{\xi^*}=\int_a^b\left|g_{\xi^*}(u)\right|du
  =
 \frac 12\int_0^{|\xi^*_1|}\omega(2u)du+
 2\frac 12\sum\nolimits_{i=2}^{n}\int_0^{|\xi^*_i|/2}\omega(2u)du
 \\  +\frac 12\int_0^{|\xi^*_{n+1}|}\omega(2u)du
 \ge 
 \int_0^{(|\xi^*_1|+|\xi^*_{n+1}|)/2}\omega(2u)du
  +
 \sum\nolimits_{i=2}^{n}\int_0^{|\xi^*_i|/2}\omega(2u)du
  \\  = 
  \frac{1}2\int_0^{|\xi^*_1|+|\xi^*_{n+1}|}\omega(u)du+\frac{1}2\sum\nolimits_{i=2}^{n}\int_0^{|\xi^*_i|}\omega(u)du
  \ge \frac n2\int_0^{(b-a)/n}\omega(u)du.
  \end{gather*}
  The function $G_{\xi^*}$ is monotone on the segments $[a,\eta_1(\xi^*)], [\eta_i(\xi^*),\eta_{i+1}(\xi^*)]$, $i = 1,\ldots, n-1$, $[\eta_n(\xi^*),b]$ and $G_{\xi^*}(0) = G_{\xi^*}(b) = 0$. Hence
  $$
   2n\max\limits_i |G_{\xi^*}(\eta_i(\xi^*))|
   \ge
   2\sum\nolimits_{i=1}^n |G_{\xi^*}(\eta_i(\xi^*))|
   =\bigvee\nolimits_a^bG_{\xi^*}
   \geq
   \frac n2\int_0^{(b-a)/n}\omega(u)du,
  $$
  which implies~\eqref{splineNorm} for $f_{\bf t}=G_{\xi^*({\bf t})}$. If ${\bf t}$ is the uniform partition, all inequalities above become equalities, and hence~\eqref{splineNorm} also becomes equality.   
  The lemma is proved.\end{proof}

\subsection{Optimal recovery of the identity operator.}
Using the definition of the class $W^1H^\omega ([a,b],X)$, for an isotropic $L$-space $X$, $f\in W^1H^\omega ([a,b],X)$ and $t\in [a,b]$, applying Theorem~\ref{th::ostrowskiInequality} to  $\DerH f\in H^\omega ([a,b],X)$, one has
\begin{multline}\label{*}
h_X\left( f(t),\frac{b-t}{b-a}f(a)+\frac{t-a}{b-a}f(b)\right)
=
h_X\left( f(a)+\int_a^t\DerH f(u)du,\right.
\\ 
\left.\frac{b-t}{b-a}f(a)+\frac{t-a}{b-a}f(a)
+
\frac{t-a}{b-a}\int_a^b\DerH f(u)du\right)
\\ =
h_X\left(\int_a^t\DerH f(u)du,\right.
\left.\frac{t-a}{b-a}\int_a^b\DerH f(u)du\right)
\le 
\frac{(b-t)(t-a)}{(b-a)^2}\int_0^{b-a}\omega(u)du.
\end{multline}

Next we apply the obtained inequality to prove an estimate of the deviation of a function $f\in W^1H^\omega ([a,b],X)$ at a fixed point $t\in [a,b]$ from the interpolation polygonal function. Let a partition ${\bf t}$ as in~\eqref{partition} be given.
The interpolation polygonal function is 
\begin{equation}\label{interpolationPolyline}
 l_f({\bf t};t)=\frac{t_{k+1}-t}{t_{k+1}-t_k}f(t_k)+\frac{t-t_k}{t_{k+1}-t_k}f(t_{k+1}),\; t\in[t_k,t_{k+1}].  
\end{equation}
Applying~\eqref{*}, we obtain that for $t\in [t_k, t_{k+1}]$
\begin{equation}\label{polylineUpperEstimate}
h_X(f(t),l_f({\bf t};t))\le \frac{(t_{k+1}-t)(t-t_k)}{(t_{k+1}-t_k)^2}\int_0^{t_{k+1}-t_k}\omega(u)du.    
\end{equation}
Therefore for a uniform partition ${\bf t}^*$ of the segment $[a,b]$ the following generalization of a result by Malozemov~\cite{Malozemov66} holds: for each $f\in W^1H^\omega ([a,b],X)$
\begin{equation}\label{polylineDeviation}
\max\limits_{t\in [a,b]}h_X(f(t),l_f({\bf t}^*;t))\le \frac 1{4}\int_0^{(b-a)/n}\omega(u)du.
\end{equation}

\begin{theorem} If ${\bf t}$ is a partition of $[a,b]$, 
$
I_{\bf t}(f)=(f(t_0),f(t_1),\ldots , f(t_n))
$
is the information operator and ${\rm Id}$ is the identity operator, then for an isotropic $L$-space $X$
$$
\inf\limits_{\bf t}\mathcal{E}({\rm Id}, W^1H^\omega([a,b],X),I_{\bf t}, C([a,b],X))
=
\frac 14\int_0^{(b-a)/n}\omega(t)dt.
$$
The optimal information operator is $I_{{\bf t}^*}$ where ${\bf t}^*$ is the uniform partition, and the optimal method of recovery is $\Phi(I_{{\bf t}^*}(f)) =l_f({\bf t}^*)$, where $l_f({\bf t})$ is defined by~\eqref{interpolationPolyline}.
\end{theorem}

\begin{proof} 
For arbitrary partition ${\bf t}$ let $f_{\bf t}$ be the function from Lemma~\ref{l::omegaSpline}, and $x\in X^c\cap X^{\rm inv}$, $h_X(x,\theta) = 1$.
Using Lemmas~\ref{l::errorOfRecoveryFromBelow} and~\ref{l::distBetweenInverseElems}, and isotropness of $X$, we obtain
$$
\mathcal{E}({\rm Id}, W^1H^\omega([a,b],X),I_{\bf t}, (C[a,b],X))
 \ge
{ \frac 12\max\limits_{t\in [a,b]}h_X\left((f_{\bf t})_x(t), (f_{\bf t})_{x'}(t)\right)}
 $$
 $$
 = \frac 12 \max\limits_{t\in [a,b]}h_X(2(f_{\bf t})_+(t)\cdot x,2(f_{\bf t})_-(t)\cdot x) =
\max\limits_{t\in [a,b]}|f_{\bf t}\left(t\right)|
\geq 
\frac 14\int_0^{(b-a)/n}\omega(t)dt.
$$

It follows from~\eqref{polylineDeviation}, that in the case of the uniform partition we have equalities in the above inequalities.
The theorem is proved.\end{proof}
\subsection{Recovery of the derivative}
Consider the problem about the deviation of the Hukuhara type derivative of a function $f\in W^1H^\omega ([a,b], X)$ from the derivative of its interpolation polygonal function. 

The Hukuharu type derivative of the interpolation at the points of the partition ${\bf t}$ polygonal function $l_f({\bf t})$ for $t\in (t_k,t_{k+1})$ is equal to 
$$
\DerH l_f({\bf t};t)=\frac{f(t_{k+1})\DiffH f(t_k)}{t_{k+1}-t_k}=\frac1{t_{k+1}-t_k}\int_{t_k}^{t_{k+1}}\DerH f(u)du, \;k = 0,\ldots, n-1.
$$
We define it at the points $t_k$, setting
$$
 \DerH l_f({\bf t};t_k)=
\begin{cases}
   (f(t_{k+1}\DiffH f(t_k))/(t_{k+1}-t_k), & \text{  if } k=0,1,\ldots,n-1,\\
      (f(t_n)\DiffH f(t_{n-1}))/(t_{n}-t_{n-1}), & \text{ if } k=n.
\end{cases}
$$
For $t\in [t_k, t_{k+1}]$ we obtain, using Corollary~\ref{c::valueIntegralDeviation} 
\begin{equation}\label{derivativeDeviation}
h_X\left(\DerH f(t), \DerH l_f({\bf t};t)\right)\le \frac 1{t_{k+1}-t_k}\int_{t_k}^{t_{k+1}}\omega (|s-t|)ds.    
\end{equation}
The following theorem generalizes the results from~\cite{Malozemov67}.
\begin{theorem}
Let $\omega$ be an arbitrary modulus of continuity and ${\bf t^*}=(t^*_0,\ldots ,t^*_n)$ be the uniform partition of the segment $[a,b]$. Then 
$$
\mathcal{E}(\DerH, W^1H^\omega([a,b],X),I_{{\rm\bf t}^*}, B([a,b],X))=\frac n{b-a}\int_{0}^{(b-a)/n}\omega (u)du.
$$
The optimal method of recovery is
$$
\Phi (f(t_0^*),f(t_1^*),\ldots ,f(t_n^*))=\DerH l_f({\bf t}^*).
$$
\end{theorem}
\begin{proof}
From~\eqref{derivativeDeviation} it follows that 
$$
\sup\limits_{f\in W^1H^\omega ([a,b], X)}\sup\limits_{t\in [a,b]}h_X\left(\DerH f(t), \DerH l_f({\bf t}^*,t)\right)
\leq 
\frac n{b-a}\int_{0}^{(b-a)/n}\omega (u)du.
$$
An extremal function is built as follows. Set
$
g_0(t)=\min\limits_{k\colon 2k\le n}\omega(|t-t_{2k}^*|)
$
and
$$
g(t)=g_0(t)-\frac{1}{b-a}\int_a^bg_0(u)du.
$$
The function $f_{{\bf t}^*}(t):=\int_a^tg(u)du$ belongs to $W^1H^\omega([a,b],\RR)$. Moreover, since ${\bf t}^*$ is the uniform partition, $f_{{\bf t}^*}(t_k) = 0$, $k = 0,\ldots, n$, and hence $l_f({\bf t}^*)\equiv 0$. Finally, applying Lemma~\ref{l::errorOfRecoveryFromBelow} to functions
$(f_{{\bf t}^*})_x$ and $(f_{{\bf t}^*})_{x'}$ ($x\in X^c\cap X^{\rm inv}$, $h_X(x,\theta)=1$)
we obtain 
\begin{gather*}
\mathcal{E}(\DerH, W^1H^\omega([a,b],X),I_{{\rm\bf t}^{*}}, B([a,b],X))
\geq  
\frac 12h_X\left(\DerH(f_{{\bf t}^*})_x(a), \DerH(f_{{\bf t}^*})_{x'}(a)\right)
\\ =
 |f'_{{\bf t}^*}(a)|
 =
 \frac{1}{b-a}\int_a^bg_0(u)du
 =
 \frac n{b-a}\int_{0}^{(b-a)/n}\omega (u)du.
 \end{gather*}
The theorem is proved.\end{proof}

\section{On Inequalities of Landau type and Stechkin's Problem for Hukuharu Type Divided Differences and Derivatives}\label{s::StechkinPr}
\subsection{Deviation of Hukuhara type divided differences and derivatives}
Let $t\in [a,b]$ and non-negative numbers $\gamma_1,\gamma_2,h_1,h_2$ such that 
\begin{equation}\label{h,gamma}
\gamma_1+\gamma_2>0, h_1+h_2>0, \text{ and } 
[t-\gamma_1,t+\gamma_2]\subset [t-h_1,t+h_2]\subset [a,b]
\end{equation}
be given. For a function $f\in W^1H^\omega([a,b],X)$ set
$$
\Delta^H_{\gamma_1,\gamma_2}f(t)=\frac{f(t+\gamma_2)\DiffH f(t-\gamma_1)}{\gamma_1+\gamma_2}.
$$
Applying Theorem~\ref{th::ostrowskiInequality} to the segments $[t-\gamma_1,t+\gamma_2]$ and $[t-h_1,t+h_2]$, and writing 
$I(\alpha)$ instead of $I(0,\alpha)$, we obtain
$$
h_X(\Delta^H_{\gamma_1,\gamma_2}f(t),\Delta^H_{h_1,h_2}f(t))
=h_X\left(\frac 1{\gamma_1+\gamma_2}\int_{t-\gamma_1}^{t+\gamma_2}\DerH f(u)du,\frac 1{h_1+h_2} \int_{t-h_1}^{t+h_2}\DerH f(u)du\right)
$$
$$
\le \frac{(h_1-\gamma_1)+(h_2-\gamma_2)}{(h_1+h_2)^2}\left\{ I\left(\frac{(h_1+h_2)(h_1-\gamma_1)}{(h_1-\gamma_1)+(h_2-\gamma_2)}\right)\right.
+\left.I\left(\frac{(h_1+h_2)(h_2-\gamma_2)}{(h_1-\gamma_1)+(h_2-\gamma_2)}\right)\right\}
$$
$$
=:K(\gamma_1,\gamma_2;h_1,h_2).
$$
If $\omega$ is a concave modulus of continuity, then the estimate 
\begin{equation}\label{a}
h_X(\Delta^H_{\gamma_1,\gamma_2}f(t),\Delta^H_{h_1,h_2}f(t))\le K(\gamma_1,\gamma_2;h_1,h_2)
\end{equation}
is sharp. Extremal functions can be built as follows. Start with the extremal function $g$ from Theorem~\ref{th::ostrowskiInequality} for the segments $[t-\gamma_1,t+\gamma_2]$ and $[t-h_1,t+h_2]$. Continue it setting \begin{equation}\label{continuation}
    g(u)=g(t-h_1) \text{ for } u\le t-h_1 \text{ and }  g(u)=g(t+h_2) \text{ for } t\ge t+h_2. 
\end{equation}
Inequality~\eqref{a} becomes equality on the functions
$
f(u)=\int_a^ug(s)ds+y, \; u\in [a,b]$, $y\in X^{\rm c}.
$
Shrinking the segment $[t-\gamma_1, t+\gamma_2]$ into the point $t$, we obtain
\begin{equation}\label{a'}
h_X(\DerH f(t),\Delta_{h_1,h_2}^Hf(t))\le\frac {I(h_1)+I(h_2)}{h_1+h_2}.    \end{equation}
This inequality is sharp for arbitrary modulus of continuity $\omega$. Extremal functions can be built analogously to the extremal ones for~\eqref{a}, except we need to start from the extremal function from Corollary~\ref{c::valueIntegralDeviation}.

\subsection{Landau type inequalities}
Below for brevity we write $\overline{W}^1H^\omega([a,b],X):=\bigcup_{k>0}k\cdot {W}^1H^\omega([a,b],X)$
$$
\|x\|_X=h_X(x,\theta), \;
\| f\|_{\omega,X}=\sup\limits_{\substack{t',t''\in [a,b] \\ t'\neq t''}}\frac{h_X(f(t'),f(t''))}{\omega(|t'-t''|)},\;\;\;
\| f\|_{C([a,b],X)}=\sup\limits_{t\in [a,b]}\| f(t)\|_X.
$$
\begin{theorem}\label{1stLandauIn}
Let $\omega$ be a modulus of continuity, and $X$ be an isotropic $L$-space. For all  $t\in [a,b]$, non-negative $\gamma_1,\gamma_2, h_1,h_2$ that satisfy~\eqref{h,gamma}, and $f\in \overline{W}^1H^\omega([a,b],X)$,
\begin{equation}\label{b}
\| \Delta_{\gamma_1,\gamma_2}^Hf(t)\|_X\le K(\gamma_1,\gamma_2; h_1,h_2)\| \DerH f\|_{\omega,X}+\| \Delta_{h_1,h_2}^Hf(t)\|_X,   
\end{equation}
\begin{equation}\label{c}
    \|\DerH f(t)\|_X\le \frac {I(h_1)+I(h_2)}{h_1+h_2}\| \DerH f\|_{\omega,X}+\| \Delta_{h_1,h_2}f(t)\|_X.
\end{equation}
Inequality~\eqref{b} is sharp for concave $\omega$. Inequality~\eqref{c} is sharp for arbitrary $\omega$.
\end{theorem}
\begin{proof} Inequalities~\eqref{b} and~\eqref{c} follow from~\eqref{a} and~\eqref{a'} respectively. An extremal function for~\eqref{b} can be built as follows. Let $g$ be a non-negative extremal function in Theorem~\ref{th::ostrowskiInequality} for the case $X=\RR$ and the segments  $[t-h_1,t+h_2]$, $[t-\gamma_1,t+\gamma_2]$. Continue it to the whole segment $[a,b]$ by~\eqref{continuation}.
Note, that due to construction of $g$, we can assume that there exists $\gamma\in (t-\gamma_1,t+\gamma_2)$ such that $g$ increases on $(t-h_1,\gamma)$ and decreases on $(\gamma,t+h_2)$. Hence 
$$
\frac{1}{\gamma_1+\gamma_2}\int_{t-\gamma_1}^{t+\gamma_2}g(u)du \geq \frac{1}{h_1+h_2}\int_{t-h_1}^{t+h_2}g(u)du,
$$
and the function $f(u) = \int_{a}^ug(s)ds$  turns inequality~\eqref{b} into equality in the case $X = \RR$. Indeed, 
$$
 \Delta_{\gamma_1,\gamma_2}^Hf(t)=\left(\frac{1}{\gamma_1+\gamma_2}\int_{t-\gamma_1}^{t+\gamma_2}g(u)du - \frac{1}{h_1+h_2}\int_{t-h_1}^{t+h_2}g(u)du\right)+\frac{1}{h_1+h_2}\int_{t-h_1}^{t+h_2}g(u)du
$$
$$
=K(\gamma_1,\gamma_2,h_1,h_2)+\Delta_{h_1,h_2}^Hf(t).
$$
In general case, the function $f_x$ with $x\in X^{\rm c}$, $\| x\|_X = 1$ is extremal for inequality~\eqref{b}. 

An extremal function for~\eqref{c} can be built analogously to the one in~\eqref{b}, but we need to start from a non-negative extremal for Corollary~\ref{c::valueIntegralDeviation} for the point $t$ and the segment $[t-h_1,t+h_2]$. 
\end{proof}

\begin{theorem}\label{th::landauIneq}
Under the conditions of Theorem~\ref{1stLandauIn}
for any $f\in \overline{W}^1H^\omega([a,b],X)$,
\begin{equation}\label{d}
\| \Delta_{\gamma_1,\gamma_2}f(t)\|_X\le K(\gamma_1,\gamma_2; h_1,h_2)\| \DerH f\|_{\omega,X}+\frac 2{h_1+h_2}\| f\|_{C([a,b],X)},    \end{equation}
\begin{equation}\label{e}
\| \DerH f(t)\|_X\le \frac {I(h_1)+I(h_2)}{h_1+h_2}\|\DerH f\|_{\omega,X}+\frac 2{h_1+h_2}\| f\|_{C([a,b],X)}.
\end{equation}
If for given $t\in [a,b]$ and $h>\gamma>0$ 
\begin{equation}\label{restrictionsOnH}
    \gamma_1=\min\{ \gamma, t-a\},\; \gamma_2=\min\{ \gamma, b-t\},
\; h_1=\min\{ h, t-a\},\; h_2=\min\{ h, b-t\}
\end{equation}
and $\omega$ is concave, then inequality~\eqref{d} is sharp.  If for $t\in [a,b]$ and $h>0$
\begin{equation}\label{restrictionsOnh}
     h_1=\min\{ h, t-a\},\; h_2=\min\{ h, b-t\}, 
     \end{equation}
and $\omega$ is an arbitrary modulus of continuity, then inequality~\eqref{e} is sharp.
\end{theorem}
\begin{proof}
Inequalities~\eqref{d} and~\eqref{e} follow from~\eqref{b} and~\eqref{c}, since 
$$
\|\Delta_{h_1,h_2}f(t)\|_X\le \frac 2{h_1+h_2}\| f\|_{C([a,b],X)}.
$$
We prove their sharpness under above conditions on the numbers $t,\gamma_1,\gamma_2$, $h_1$ and $h_2$.

Let, for definiteness, $t\le (a+b)/2$, and hence $h_2\geq h_1$. For inequality~\eqref{d} as a function $g$ we take the non-negative extremal function in Theorem~\ref{th::ostrowskiInequality} for the segments $[t-\gamma_1,t+\gamma_2]$ and $[t-h_1,t+h_2]$ and $X=\RR$  such that $g(t+h_2)=0$. Continue it to the segment $[a,b]$ setting $g(u) = 0$, $u\notin [t-h_1,t+h_2]$. For inequality~\eqref{e} we take $g(s)=(\omega(h_2)-\omega(|s-t|))_+$, $s\in [a,b]$. Both functions belong to $H^\omega([a,b],\RR)$ (the first one in the case of concave $\omega$) and are non-negative on $[a,b]$.
Choose $\xi\in [t-h_1,t+h_2]$ so that 
$
\int^\xi_{t-h_1}g(u)du=\int_\xi^{t+h_2}g(u)du.
$
The function
\begin{equation}\label{segmentExtremalFunc}
   f(u)=\left(\int_\xi^{u}g(s)ds\right)_x,\; x\in X^{\rm c}\cap X^{\rm inv}, \;\| x\|_X=1 
\end{equation}
is extremal. The theorem is proved. \end{proof}

Note, that for the  extremal in inequality~\eqref{e} function one has
\begin{equation}\label{segmentExtremalFuncNorm}
 \| f\|_{C([a,b],X)}=\frac 12\int_{t-h_1}^{t+h_2}[\omega(h_2)-\omega(|s-t|)]ds=\frac{h_1+h_2}2\omega(h_2)-\frac {I(h_1)+I(h_2)}{2}.   
\end{equation}

\subsection{Approximation of operators by the ones with smaller norms}
In the space $C([a,b],X)$ consider the cone $C^H([a,b],X)$, that consists of functions $f$ such that for all $t\in [a,b]$ and $\gamma_1,\gamma_2>0$ such that $[t-\gamma_1,t+\gamma_2]\subset [a,b]$, the difference $\Delta^H_{\gamma_1,\gamma_2} f(t)$  exists. We call a positively homogeneous operator $T\colon C^H([a,b],X) \to X$ bounded, if
$$
\| T\|=\sup\{\| Tf\|_X\; :\; f\in C^H([a,b],X),\;\| f\|_{C([a,b],X)}\le 1\}<\infty .
$$
Assume that an operator  $A\colon \overline{W}^1H^\omega ([a,b],X)\to X$, a number  $N>0$ and an operator  $T\colon C^H([a,b],X)\to X$ such that $\| T\|\le N$ are given. Set
$$
U(A,T)=\sup\limits_{f\in W^1H^\omega([a,b],X)}h_X(Af,Tf).
$$
The quantity 
$$
E(A,N)=\inf\limits_{\| T\|\le N}U(A,T)
$$
is called the best approximation of the operator $A$ by operators $T$ with $\|T\|\le N$. It is clear, that if $A$ is defined on $C^H([a,b],X)$, is bounded, and $N\ge \| A\|$, then  $E(A,N)=0.$ 

 For $t\in [a,b]$ denote by $\Delta_{\gamma_1,\gamma_2}(t)$ and $\DerH(t)$ the operators that act by the formulae 
$$
\Delta_{\gamma_1,\gamma_2}(t)f=\Delta^H_{\gamma_1,\gamma_2}f(t) \text{ and } \DerH(t)f=\DerH f(t).
$$
\begin{theorem}
Let $\omega$ be a modulus of continuity, $X$ be an isotropic $L$-space, $t\in [a,b]$, and numbers $h>\gamma>0$ be given. Let also numbers $\gamma_1,\gamma_2,h_1,h_2$ be defined by~\eqref{restrictionsOnH}. If $\omega$ is concave, then
\begin{equation}\label{k}
  E\left(\Delta_{\gamma_1,\gamma_2}(t),\frac 2{h_1+h_2}\right)=U\left(\Delta_{\gamma_1,\gamma_2}(t),\Delta_{h_1,h_2}(t)\right)= K(\gamma_1,\gamma_2;h_1,h_2),
\end{equation}
and for arbitrary $\omega$
\begin{equation}\label{l}
  E\left(\DerH(t),\frac 2{h_1+h_2}\right)=U\left(\DerH(t),\Delta_{h_1,h_2}(t)\right)= \frac {I(h_1)+I(h_2)}{h_1+h_2}.  
\end{equation}
\end{theorem}
\begin{proof}  It is clear that $\| \Delta^H_{h_1,h_2}\|\le \frac 2{h_1+h_2}$. Due to~\eqref{a} and~\eqref{a'} we have
$$
E\left( \Delta_{\gamma_1,\gamma_2}(t),\frac 2{h_1+h_2}\right)\le U(\Delta_{\gamma_1,\gamma_2}(t),\Delta_{h_1,h_2}(t))\le K(\gamma_1,\gamma_2;h_1,h_2)
$$
and
$$
E\left(\DerH(t),\frac 2{h_1+h_2}\right)\le U(\DerH(t),\Delta_{h_1,h_2}(t))\le \frac {I(h_1)+I(h_2)}{h_1+h_2}.
$$
 We also proved that there exist functions $f_1,f_2\in W^1H^\omega([a,b],X)$ such that 
\begin{equation}\label{alpha}
 \| \Delta^H_{\gamma_1,\gamma_2}f_1(t)\|_X= K(\gamma_1,\gamma_2; h_1,h_2)+\frac 2{h_1+h_2}\| f_1\|_{C([a,b],X)}
\end{equation}
and
$$
\| \DerH f_2(t)\|_X=\frac {I(h_1)+I(h_2)}{h_1+h_2} +\frac 2{h_1+h_2}\| f_2\|_{C([a,b],X)}.
$$
To prove~\eqref{k}, assume there exists an operator  $T$, $\| T\|\le\frac 2{h_1+h_2}$ such that
$$
U(\Delta_{\gamma_1,\gamma_2}(t),T)< K(\gamma_1,\gamma_2;h_1,h_2).
$$
Then for the function  $f_1$ we get a strict inequality
$$
\| \Delta^H_{\gamma_1,\gamma_2}f_1(t)\|_X< K(\gamma_1,\gamma_2; h_1,h_2)+\frac 2{h_1+h_2}\| f_1\|_{C([a,b],X)}, 
$$
which contradicts to~\eqref{alpha}. Equality~\eqref{l} can be proved similarly.  \end{proof}

\subsection{Recovery of an operator given inexact data}
Finally, we consider the problem of optimal recovery of an operator  $A$ on the elements of the class $W^1H^\omega ([a,b],X)$ known with error. For an operator $A$, bounded operator  $T$ and a number $\delta>0$ set
$$
U_\delta (A,T)=\sup\{ h_X(Af,Tg)\colon f\in W^1H^\omega ([a,b],X), g\in C([a,b],X), h_{C([a,b],X)}(f,g)\le \delta \}.
$$
The problem is to find the quantity
$$
\mathcal{E}_\delta(A)=\inf\nolimits_TU_\delta(A,T)
$$
and the operator $T^*$ on which the infimum in the right-hand side of the equality is attained. 
\begin{theorem}
Let $\omega$ be a modulus of continuity,  $t\in [a,b]$, $h>0$, and $\DerH(t) f=\DerH f(t)$ for $f\in W^1H^\omega ([a,b],X)$. If the numbers $h_1,h_2$ are defined by~\eqref{restrictionsOnh}, and 
$$ \delta =\frac {h_1+h_2}2\max\{\omega(h_1), \omega(h_2)\}-\frac {I(h_1)+I(h_2)}2,
$$
then for the operator $\Delta_{h_1,h_2}(t)f=\Delta^H_{h_1,h_2}f(t)$ we have
$$
\mathcal{E}_\delta(\DerH(t))=U_\delta(\DerH(t),\Delta_{h_1,h_2}(t)) = \max\{\omega(h_1), \omega(h_2)\}.
$$
\end{theorem}
\begin{proof}
For each $f\in W^1H^\omega ([a,b],X)$, $g\in C([a,b],X)$ such that $h_{C([a,b],X)}(f,g)\leq \delta$, due to~\eqref{a'},
$$
h_X(\DerH f(t) , \Delta^H_{h_1,h_2}g(t))\le h_X(\DerH f(t) , \Delta^H_{h_1,h_2}f(t))+h_X(\Delta^H_{h_1,h_2}f(t) , \Delta^H_{h_1,h_2}g(t))
$$
$$
\le \frac {I(h_1)+I(h_2)}{h_1+h_2} +\frac 2{h_1+h_2}\delta = \max\{\omega(h_1), \omega(h_2)\}.
$$
Hence,
$
\mathcal{E}_\delta(\DerH(t))\le\max\{\omega(h_1), \omega(h_2)\}.
$
On the other hand, for the function $f$ defined by~\eqref{segmentExtremalFunc}, due to~\eqref{segmentExtremalFuncNorm},
$$
\mathcal{E}_\delta(\DerH(t))\ge\| \DerH f(t)\|_X= \frac {I(h_1)+I(h_2) }{h_1+h_2}+\frac 2{h_1+h_2}\| f\|_{C([a,b],X)}=\max\{\omega(h_1), \omega(h_2)\}
$$
and the theorem is proved. \end{proof}

\bibliographystyle{plain}
\bibliography{bibliography}

\begin{thebibliography}{10}

\bibitem{Anastassiou03}
G.~A. Anastassiou.
\newblock Fuzzy {O}strowski type inequalities.
\newblock {\em Comput. Appl. Math.}, 22(2):279--292, 2003.

\bibitem{Anastassiou12}
G.~A. Anastassiou.
\newblock Ostrowski and {L}andau inequalities for {B}anach space valued
  functions.
\newblock {\em Mathematical and Computer Modelling}, 55(3):312 -- 329, 2012.

\bibitem{ANASTASSIOU2012312}
G.~A. Anastassiou.
\newblock Ostrowski and landau inequalities for banach space valued functions.
\newblock {\em Math. Comput. Model.}, 55(3):312 -- 329, 2012.

\bibitem{anastassiou2010}
G.A. Anastassiou.
\newblock {\em Fuzzy Mathematics: Approximation Theory}.
\newblock Studies in Fuzziness and Soft Computing. Springer, 2010.

\bibitem{Arestov}
V.~V. Arestov.
\newblock Approximation of unbounded operators by bounded operators and related
  extremal problems.
\newblock {\em Russian Math. Surveys}, 51(6):1093--1126, 1996.

\bibitem{Aseev}
S.~M. Aseev.
\newblock Quasilinear operators and their application in the theory of
  multivalued mappings.
\newblock {\em Proc. Steklov Inst. Math.}, 167:23--52, 1986.

\bibitem{Aubin}
J.~P. Aubin and H.~Frankowska.
\newblock {\em Set-valued analysis}.
\newblock Birkh auser, Boston, 1990.

\bibitem{babenko19}
V.~Babenko.
\newblock Calculus and nonlinear integral equations for functions with values
  in l-spaces.
\newblock {\em Anal Math}, 45:727--755, 2019.

\bibitem{VeraBabenko_JANO}
V.~F. Babenko and V.~V. Babenko.
\newblock Best approximation, optimal recovery, and {L}andau inequalities for
  derivatives of {H}ukuhara-type in function {L}-spaces.
\newblock {\em J. Appl. Numer. Optim.}, 1:167--182, 2019.

\bibitem{Babenko15}
V.~F. Babenko, V.~V. Babenko, and M.~V. Polischuk.
\newblock On the optimal recovery of integrals of set-valued functions.
\newblock {\em Ukrainian Math. J.}, 67(9):1306--1315, 2016.

\bibitem{BKKP}
V.~F. Babenko, N.~P. Korneichuk, V.~A. Kofanov, and S.~A. Pichugov.
\newblock {\em Inequalities for Derivatives and Their Applications}.
\newblock Naukova Dumka, 2003.
\newblock (in Russian).

\bibitem{bagdasarov2012}
S.~Bagdasarov.
\newblock {\em Chebyshev Splines and Kolmogorov Inequalities}.
\newblock Operator Theory: Advances and Applications. Birkh{\"a}user Basel,
  2012.

\bibitem{barnet}
N.~S. Barnet, P.~Cerone, A.~M. Dragomir, and Fink~A. M.
\newblock Comparing two integral means for absolutely continuous mappings whose
  derivatives are in ${L}_{\infty} [a,b]$ and applications.
\newblock {\em Computers \& Mathematics with Applications}, 44(1):241 -- 251,
  2002.

\bibitem{Barnett01}
N.~S. Barnett, C.~Buse, P.~Cerone, and S.~S. Dragomir.
\newblock On weighted {O}strowski type inequalities for operators and
  vector-valued functions.
\newblock {\em Journal of Inequalities in Pure and Applied Mathematics},
  3(1):1--21, 2002.

\bibitem{Barnett02}
N.~S. Barnett, C.~Buse, P.~Cerone, and S.~S. Dragomir.
\newblock Ostrowski's inequality for vector-valued functions and applications.
\newblock {\em Computers and Mathematics with Applications}, 44(5--6):559--572,
  2002.

\bibitem{Borisovich}
Yu.~G. Borisovich, B.~D. Gel'man, A.~D. Myshkis, and V.~V. Obukhovskii.
\newblock Multivalued mappings.
\newblock {\em J. Sov. Math.}, 24(6):719--791, 1984.

\bibitem{Borodachov98}
S.~V. Borodachov.
\newblock On optimization of interval quadrature formulae on some nonsymmetric
  classes of periodic functions.
\newblock {\em Bull. Dnepropetrovsk Univ. Math.}, 4:19--24, 1999.
\newblock (in Russian).

\bibitem{Chalco12}
Y.~Chalco-Cano, A.~Flores-Franulic, and H.~Roman-Flores.
\newblock Ostrowski type inequalities for interval-valued functions using
  generalized {H}ukuhara derivative.
\newblock {\em Comput. Appl. Math.}, 31(2):457--472, 2012.

\bibitem{Chalco15}
Y.~Chalco-Cano and W.~A. Lodwick.
\newblock Ostrowski type inequalities and applications in numerical integration
  for interval-valued functions.
\newblock {\em Soft Comput}, 19:3293--3300, 2015.

\bibitem{Diamond2}
P.~Diamond and P.~Kloeden.
\newblock {\em Metric Spaces of Fuzzy Sets: Theory and Applications}.
\newblock World Scientific Publishing Company, 1994.

\bibitem{Dragomir03}
S.~S. Dragomir.
\newblock A weighted {O}strowski type inequality for functions with values in
  hilbert spaces and applications.
\newblock {\em J Korean Math Soc}, 40(2):207--224, 2003.

\bibitem{Dragomir17}
S.~S. Dragomir.
\newblock Ostrowski type inequalities for {L}ebesgue integral: a survey of
  recent results.
\newblock {\em Australian J. Math. Anal. Appl.}, 14(1):1--–287, 2017.

\bibitem{Drozhina}
L.~V. Drozhzhina.
\newblock On quadrature formulas for random processes.
\newblock {\em Dopovidi Akad. Nauk Ukrain. RSR, Ser. A}, 9:775--777, 1975.
\newblock (in Ukrainian).

\bibitem{nira2014}
N.~Dyn, E.~Farkhi, and A.~Mokhov.
\newblock {\em Approximation Of Set-valued Functions: Adaptation Of Classical
  Approximation Operators}.
\newblock World Scientific Publishing Company, 2014.

\bibitem{Guessab02}
A.~Guessab and G.~Schmeisser.
\newblock Sharp integral inequalities of the {H}ermite–{H}adamard type.
\newblock {\em J. Approx. Theory}, 115(2):260--288, 2002.

\bibitem{hille}
E.~Hille and R.S. Phillips.
\newblock {\em Functional Analysis and Semi Groups}.
\newblock Amer. Math. Soc. Colloq. Publ. AMS, 1957.

\bibitem{Hukuhara}
M.~Hukuhara.
\newblock Integration des applications mesurables dont la valeur est un compact
  convexe.
\newblock {\em Funkcial. Ekvac.}, 10:205--223, 1967.

\bibitem{Korneichuk59}
N.~P. Korneichuk.
\newblock Approximation of periodic functions satisfying {L}ipschitz's
  condition by {B}ernstein-{R}ogosinski's sums.
\newblock {\em Dokl. Akad. Nauk SSSR}, 125:258--261, 1959.

\bibitem{Korneichuk61}
N.~P. Korneichuk.
\newblock On the degree of approximation of functions of class ${H}^{(a)}$ by
  means of trigonometric polynomials.
\newblock In {\em Studies of Modern Problems of Constructive Theory of
  Functions}, pages 148--154. Moscow, Fizmathgiz, 1961.

\bibitem{Korneichuk62}
N.~P. Korneichuk.
\newblock Extremal properties of periodic functions.
\newblock {\em Dopovidi Akad. Nauk Ukrain. RSR}, pages 993--998, 1962.

\bibitem{Korneichuk68}
N.~P. Korneichuk.
\newblock Best cubature formulas for some classes of functions of many
  variables.
\newblock {\em Mathematical Notes of the Academy of Sciences of the USSR},
  3:360--367, 1968.

\bibitem{Korneichuk71}
N.~P. Korneichuk.
\newblock Extremal values of functionals and the best approximation on classes
  of periodic functions.
\newblock {\em Math. USSR-Izv.}, 1971(1):97--129, 1971.

\bibitem{SplinesInApproxTh}
N.~P. Korneichuk.
\newblock {\em Splines in approximation theory}.
\newblock Moscow, 1984.
\newblock (in Russian).

\bibitem{ExactConstants}
N.~P. Korneichuk.
\newblock {\em Exact Constants in Approximation Theory}.
\newblock Encyclopedia of Mathematics and its Applications. Cambridge
  University Press, 1991.

\bibitem{Kovalenko20}
O.~Kovalenko.
\newblock On optimal recovery of integrals of random processes.
\newblock {\em J. Math. Anal. Appl.}, 487(1):123949, 2020.

\bibitem{Kumar}
P.~Kumar.
\newblock The {O}strowski type moment integral inequalities and moment-bounds
  for continuous random variables.
\newblock {\em Comp. Math. Appl.}, 49(11):1929 -- 1940, 2005.

\bibitem{Malozemov66}
V.~N. Malozemov.
\newblock On deviation of polygonal functions.
\newblock {\em Bulletin of Leningrad University}, 7:150--153, 1966.
\newblock (in Russian).

\bibitem{Malozemov67}
V.~N. Malozemov.
\newblock On polygonal interpolation.
\newblock {\em Math. Notes}, 1(5):355--357, 1967.

\bibitem{Mitrinovich}
D.~S. Mitrinovi$\acute{c}$, J.~E. Pe$\check{c}$ari$\acute{c}$, and A.~M Fink.
\newblock {\em Inequalities Involving Functions and Their Integrals and
  Derivatives}.
\newblock Mathematics and its Applications. Kluwer Academic Publishers, 1991.

\bibitem{Nikolsky46}
S.~M. Nikol'skii.
\newblock Fourier series of functions with a given modulus of continuity.
\newblock {\em Dokl. Akad. Nauk SSSR}, 52(3):191--194, 1946.

\bibitem{Ostrowski38}
A.~Ostrowski.
\newblock Uber die {A}bsolutabweichung einer differentienbaren {F}unktionen von
  ihren {I}ntegralmittelwert.
\newblock {\em Comment. Math. Hel}, 10:226--227, 1938.

\bibitem{Roman18}
H.~Roman-Flores, Y.~Chalco-Cano, and W.~A. Lodwick.
\newblock Some integral inequalities for interval-valued functions.
\newblock {\em Comp. Appl. Math.}, 37:1306--1318, 2018.

\bibitem{Stechkin}
S.~B. Stechkin.
\newblock Best approximation of linear operators.
\newblock {\em Math. Notes}, 1(2):91--99, 1967.

\bibitem{stepanets2018}
A.~I. Stepanets.
\newblock {\em Uniform Approximations by Trigonometric Polynomials}.
\newblock De Gruyter, 2018.

\bibitem{TraubWozhniakowski}
J.~F. Traub and H.~Wo\'{z}niakowski.
\newblock {\em A general theory of optimal algorithms}.
\newblock Academic Press, 1980.

\bibitem{Vahrameev}
S.~A. Vahrameev.
\newblock {\em Applied Mathematics and Mathematical Software of Computers}.
\newblock M.: MSU Publisher, 1980.
\newblock (in Russian).

\bibitem{XIAO}
Y.~Xiao.
\newblock Landau type inequalities for banach space valued functions.
\newblock {\em J. Math. Inequal.}, 7(1):103–114, 2013.

\bibitem{Zhao19}
D.~Zhao, T.~An, G.~Ye, and W.~Liu.
\newblock Some integral inequalities for interval-valued functions.
\newblock {\em Fuzzy Sets and Systems}, (in press), 2019.

\bibitem{Zhensikbaev03}
A.~A. Zhensykbaev.
\newblock {\em Problems of recovery of operators}.
\newblock Moscow-Izhevsk, 2003.

\end{thebibliography}

\end{document}